\documentclass{amsproc}

\newtheorem{theorem}{Theorem}[section]
\newtheorem{lemma}[theorem]{Lemma}
\newtheorem{proposition}[theorem]{Proposition}

\theoremstyle{definition}
\newtheorem{definition}[theorem]{Definition}
\newtheorem{example}[theorem]{Example}

\theoremstyle{remark}
\newtheorem{remark}[theorem]{Remark}
\numberwithin{equation}{section}
\newcommand*{\bigchi}{\mbox{\Large$\chi$}}

\begin{document}

\title[Stability of quaternion matrix polynomials]
{Stability of quaternion matrix polynomials}

	\author[Pallavi]{Pallavi Basavaraju}
\address{Pallavi Basavaraju\\
	Department of Mathematics\\
	Dr. G. Shankar Government Women's First Grade College and P.G. Study Centre\\
	Ajjarakadu, Udupi, Karnataka -- 576101, Karnataka, India}
\email{pallavipoorna20@iisertvm.ac.in, pallavipoorna6@gmail.com}

\author[Shrinath]{Shrinath Hadimani}
\address{Shrinath Hadimani\\
	Department of Mathematics\\
	Manipal Institute of Technology, Manipal 
	Academy of Higher Education\\
	Manipal -- 576104, Karnataka, India}
\email{srinathsh3320@iisertvm.ac.in, shrinath.hadimani@manipal.edu}

\author[Sachindranath]{Sachindranath Jayaraman}
\address{Sachindranath Jayaraman\\
	School of Mathematics\\ 
	Indian Institute of Science Education and Research Thiruvananthapuram\\ 
	Maruthamala P.O., Vithura, Thiruvananthapuram -- 695 551, Kerala, India.}
\email{sachindranathj@iisertvm.ac.in, 
	sachindranathj@gmail.com}

%\date{January 1, 1994 and, in revised form, June 22, 1994.}


\begin{abstract}
A right quaternion matrix polynomial is an expression of the form 
$P(\lambda)= \displaystyle \sum_{i=0}^{m}A_i \lambda^i$, where $A_i$'s are 
$n \times n$ quaternion matrices with $A_m \neq 0$. The aim of this manuscript is to 
determine the location of right eigenvalues of $P(\lambda)$ relative to certain subsets 
of the set of quaternions. In particular, we extend the notion of (hyper)stability of 
complex matrix polynomials to quaternion matrix polynomials and obtain location 
of right eigenvalues of $P(\lambda)$ using the following methods: 
$(1)$ we give a relation between (hyper)stability of a quaternion matrix polynomial and 
its complex adjoint matrix polynomial, $(2)$ we prove that $P(\lambda)$ is stable 
with respect to an open (closed) ball in the set of quaternions, centered at a complex 
number if and only if it is stable with respect to its intersection with the set of 
complex numbers and $(3)$ as a consequence of $(1)$ and $(2)$, we prove that right 
eigenvalues of $P(\lambda)$ lie between two concentric balls of specific radii in the 
set of quaternions centered at the origin. A generalization of the 
Enestr{\"o}m-Kakeya theorem to quaternion matrix polynomials is obtained as an 
application. We identify classes of quaternion matrix polynomials for which stability and hyperstability are equivalent. We finally deduce hyperstability of certain univariate 
quaternion matrix polynomials via stability of certain multivariate quaternion matrix polynomials.  
\end{abstract}

\maketitle

\section{Introduction}\label{sec-1}

One of the prominent results in mathematics is the well-known fundamental 
theorem of algebra, which says any scalar polynomial with complex coefficients and 
degree $m \geq 1$ has a zero in the set of complex numbers and the number of zeros 
is equal to $m$. Although this theorem gives a precise number of zeros of a polynomial, 
when the degree of the polynomial exceeds $4$, it is difficult to calculate the 
zeros algebraically. Therefore, finding regions in which the zeros lie plays a crucial role  
in certain iterative methods \cite{Pan}. There is a vast literature on the location of zeros 
of polynomials. Some of these are \cite{Dehmer}, \cite{Geleta-Alemu}, \cite{Joyal-Labelle-Rahman}, \cite{Szabo}, \cite{Van} and references cited therein. 

\medskip

A natural generalization of scalar polynomials is to allow the coefficients to be complex 
matrices. Polynomials, whose coefficients are complex matrices, also known as matrix 
polynomials in the literature, arise in several problems in applied mathematics. An
excellent monograph on this subject is the one by Gohberg et al. \cite{Gohberg}. 
The eigenvalues of a matrix polynomial are the same as the zeros of the determinant 
(which is scalar complex polynomial) of the matrix polynomial. Therefore, one can apply 
methods given in the above cited references and scalar stability tests (see \cite{Hurwitz}, 
\cite{Marden}) to this complex polynomial to obtain the location of eigenvalues. However, 
when the size of coefficient matrices is large, computing the determinant is difficult. 
Therefore, the location of eigenvalues of complex matrix polynomials has been an 
interesting problem. The early attempts to find bounds on eigenvalues of complex matrix 
polynomials are due to Higham and Tisseur \cite{Higham-Tisseur}, by associating block 
matrices and scalar polynomials to the given matrix polynomial. Various other bounds 
obtained by associating scalar polynomials to the given matrix polynomial can be found 
in \cite{Bini-Noferini}, \cite{Le-Du-Nguyen}, \cite{Mehl-Mehrmann} and \cite{Roy-Bora}. 
Bounds for eigenvalues of matrix polynomials, very similar to those of Higham and Tisseur 
were obtained recently in \cite{Hans-Raouafi}. The readers may refer to \cite{BHMST} and 
\cite{Mehrmann-Voss} for various applications of matrix polynomials and related 
computational issues. In \cite{Basavaraju-Hadimani-Jayaraman} and \cite{Hadimani-Basavaraju-Jayaraman}, the authors have extended some of these methods 
to rational matrices.

\medskip

Recently, polynomials whose coefficients are from the noncommutative ring of quaternions 
has received interest of some mathematicians. Locating the zeros of such polynomials pose 
interesting questions due to the noncommutativity of quaternion multiplication. Readers 
may refer to the following recent articles \cite{Ahmad-Ali-2} and \cite{Mir} and the 
references cited therein. 

\medskip

For quaternion matrices as well as matrix polynomials whose coefficients have entries from 
quaternions, (henceforth known as quaternion matrix polynomials), the definition of 
the determinant is not same as that of the complex case. However, there are notions of 
determinant that have been used in the literature (see for instance, 
\cite{Pereira-Rocha} and \cite{Rodman}). Bounds on the moduli of eigenvalues of 
quaternion matrices and quaternion matrix polynomials can be determined 
by localization and perturbation theorems such as the Ger\v{s}gorin theorem, Bauer-Fike 
theorem and Ostrowski’s theorem. Bounds are also determined by associating a block matrix 
and the norm of the coefficient matrices. These results can be found in \cite{Ahmad-Ali-1}, 
\cite{Ahmad-Ali-2} and \cite{Ahmad-Ali-Ivan}.  One can also introduce notions of stability 
of (complex) matrix polynomials, which does not involve the determinant. In \cite{Oskar-Wojtylak}, the authors give two such notions for complex matrix polynomials 
to locate eigenvalues.

\medskip

This present work is an attempt to extend the notions of stability and 
hyperstability for quaternion matrix polynomials. Due to the noncommutativity of 
quaternions, these definitions are not the same as in the complex case, although 
some of the results for quaternion matrix polynomials are similar to the ones presented in 
\cite{Oskar-Wojtylak} for complex matrix polynomials. A more elaborate description of 
the results obtained in this manuscript will be given later, after introducing some 
preliminaries.
 
\section{Preliminaries}\label{sec-2}

Preliminary ideas, definitions and basic facts needed in the main results are presented in 
this section. We divide this section into two subsections. We collect basic notions in the 
complex case in the first subsection. The second one concerns notions from the ring of 
quaternions.

\subsection{The complex case}
\hspace*{\fill}\label{subsec-2.1}

$\mathbb{C}$ stands for the field of complex numbers. The space of all $n \times n$ 
matrices over $\mathbb{C}$ is denoted by $M_n(\mathbb{C})$. 
A complex matrix polynomial of degree $m$ and size $n \times n$ is a function $P 
\colon \mathbb{C} \rightarrow M_n(\mathbb{C})$ defined as 
$P(\lambda) =A_m \lambda^m + A_{m-1} \lambda^{m-1} + \cdots + A_1 \lambda + A_0$, 
where coefficients $A_0, A_1, \ldots, A_m \in M_n(\mathbb{C})$ and $A_m \neq 0$. 
If the map $\lambda \mapsto \text{det} P(\lambda)$ is not identically the zero function, 
then the matrix polynomial $P(\lambda)$ is said to be regular. 
A complex number $\lambda_0$ is called an eigenvalue of a regular matrix polynomial $P(\lambda)$, if there is a vector $y \in \mathbb{C}^n \setminus \{0\}$ such that $P(\lambda_0)y$ is a zero vector in $\mathbb{C}^n$. The vector $y \in \mathbb{C}^n$ is called an eigenvector of $P(\lambda)$ corresponding to the eigenvalue $\lambda_0$. Equivalently, if $\text{det}P(\lambda_0) = 0$, then $\lambda_0 \in \mathbb{C}$ is an eigenvalue for a given matrix polynomial $P(\lambda)$. We now introduce two important definitions that set the tone for this work. Let $\Omega$ be any nonempty subset of $\mathbb{C}$ and $P(\lambda)$ be an $n \times n$ regular 
complex matrix polynomial of degree $m$. 

\medskip
\begin{definition}\label{Def-stability-complex}
If for each nonzero vector $y \in \mathbb{C}^n$ and each $\mu \in \Omega$, there exists a 
nonzero vector $z \in \mathbb{C}^n$ such that $z^\ast P(\mu)y \neq 0$, then the matrix 
polynomial $P(\lambda)$ is said to be stable with respect to $\Omega$. 
\end{definition}

\noindent

It is easy to verify that $P(\lambda)$ has no eigenvalue in $\Omega$ if and only if $P(\lambda)$ is stable with respect to $\Omega$. A much stronger notion, namely,  hyperstability of $P(\lambda)$ is defined as follows: 

\medskip
\begin{definition}\label{Def-hyperstability-complex}
If for each nonzero vector $y \in \mathbb{C}^n$, there exists a nonzero vector 
$z \in \mathbb{C}^n$ such that $z^\ast P(\mu)y \neq 0$ for all $\mu \in \Omega$, then 
$P(\lambda)$ is said to be hyperstable with respect to $\Omega$.
\end{definition}

\noindent

Notice that if $P(\lambda)$ is hyperstable with respect to $\Omega$, 
then it is stable with respect to $\Omega$, although they are not equivalent in general 
(see for instance Example $3$ from \cite{Oskar-Wojtylak}).  Some results on hyperstability 
along with certain inequalities for matrix polynomials can also be found in 
\cite{Pikul-Szymanski-Wojtylak}.

\subsection{The quaternions setting} \label{subsec-2.2}
\hspace*{\fill}

The space of real quaternions is defined by $\mathbb{H}:= \{a_0+ a_1i+ a_2 j + a_3 k : a_i \in 
\mathbb{R}\}$ with $i^2=j^2= k^2=ijk=-1, ij=k=-ji, jk=i=-kj$ and $ki=j=-ik$. Quaternions, introduced by the Irish mathematician William Rowan Hamilton in 
1843, are an extension of the complex numbers. They play crucial roles in several areas 
such as $3D$ computer graphics, signal processing, computer visions and robotics, 
quantum mechanics and so on (see \cite{Adler} and \cite{Kuipers} for more details). For 
$q = a_0+ a_1i+ a_2 j+ a_3 k \in \mathbb{H}, \ |q|:=\sqrt{a_0^2 +a_1^2+a_2^2+a_3^2}$ 
denotes the modulus of $q$ and $\overline{q}:= a_0- a_1i -a_2 j- a_3k$ denotes the 
conjugate of $q$. Quaternions $p$ and $q$ are similar if $s^{-1}qs =p$ for some 
$s \in \mathbb{H} \setminus \{0\}$. If quaternions $p$ and $q$ are similar, then 
$|p|=|q|$. Let $M_n(\mathbb{H})$ denote the set of all $n \times n$ matrices over $\mathbb{H}$. For $A=(q_{ij}) \in M_n(\mathbb{H})$, the adjoint of $A$ is defined as 
$A^\ast = (\overline{q}_{ij})$. 

\medskip

Let $\mathbb{H}^n$ denote the set of all column vectors with entries from 
the quaternions. For $v = \begin{bmatrix}
	v_1 \\
	\vdots \\
	v_n
\end{bmatrix} \in \mathbb{H}^n$, $v_j \in \mathbb{H}$ and $\alpha \in \mathbb{H}$, let $v \alpha = 
\begin{bmatrix}
	v_1 \alpha \\
	\vdots \\
	v_n \alpha
\end{bmatrix}$. With respect to the standard addition and the scalar multiplication 
defined above, $\mathbb{H}^n$ is a right quaternion vector space over the division ring 
$\mathbb{H}$. Standard linear algebra concepts such as linear independence, 
spanning sets, basis, direct sums, and so on work in exactly the same way as for finite 
dimensional vector spaces over commutative division rings (see 
\cite{Daniel-Colombo} and \cite{Rodman} for details). The subspace spanned by 
$v_1, v_2, \ldots, v_n \in \mathbb{H}^n$ is denoted by $\text{span}_{\mathbb{H}} \{v_1, v_2, 
\ldots, v_n\} = \{v_1q_1+v_2q_2+ \cdots +v_n q_n : q_i \in \mathbb{H}\}$. 
The inner product on $\mathbb{H}^n$ is defined by 
$\langle u, v \rangle = v^* u$ for $u,v \in \mathbb{H}^n$.

\medskip

The set of quaternions being a noncommutative division ring, there are notions of 
left and right eigenvalues for matrices in $M_n(\mathbb{H})$. If 
$A \in M_n(\mathbb{H})$, a quaternion $\lambda_0$ is a right (left) eigenvalue if there 
exists a vector $y \in \mathbb{H}^n \setminus \{0\}$ such that $Ay=y\lambda_0$ ($Ay=\lambda_0 y$). It is important to note that if $\lambda_0$ is a right eigenvalue of 
$A$, any quaternion similar to $\lambda_0$ is also a right eigenvalue of $A$. Thus there 
are infinitely many right eigenvalues for $A \in M_n(\mathbb{H})$. However,  given 
$A \in M_n(\mathbb{H})$, there are exactly $n$ complex numbers with nonnegative imaginary parts that are right eigenvalues of $A$ (Theorem $5.4$, \cite{Zhang}). These 
right eigenvalues are called the standard eigenvalues of $A$. Every quaternion that is a 
right eigenvalue of $A$ falls within one of the equivalence classes of these standard eigenvalues. In contrast, determining the number of left eigenvalues and their equivalence classes remains an open question. Unlike right eigenvalues, a quaternion similar to a left 
eigenvalue need not be a left eigenvalue. In \cite{Liping-Wasin}, the authors prove 
that for a $2 \times 2$ quaternion matrix the number of left eigenvalues is either 
at most $2$ or infinite and for an $n \times n$ complex matrix the number of left 
quaternion eigenvalues is either at most $n$ or infinite. However, it remains unknown 
if an $n \times n$ quaternion matrix with a finite number of left eigenvalues has at 
most $n$ elements. Thus, left eigenvalues have received less attention in the literature 
than right eigenvalues.

\medskip

We now introduce quaternion matrix polynomials. Due to noncommutativity 
of $\mathbb{H}$ there are notions of right and left quaternion matrix polynomials.

\begin{definition}
An $n \times n$ right quaternion matrix polynomial of degree $m$ is a function 
$P \colon \mathbb{H} \rightarrow M_n(\mathbb{H})$, where $P(\lambda) = 
\displaystyle \sum_{i=0}^{m}A_i \lambda^i$, with $A_i \in M_n(\mathbb{H})$ and 
$A_m \neq 0$.	
\end{definition}

Note that the indeterminate $\lambda$ appears on the right side of the matrix coefficients. 
We now define eigenvalues of a right quaternion matrix polynomial.

\begin{definition}
A quaternion $\lambda_0 \in \mathbb{H}$ is a right (left) eigenvalue of $P(\lambda)$ 
if there exists a nonzero vector $y \in \mathbb{H}^n$ such that 
$\displaystyle \sum_{i=0}^{m}A_i  y\lambda_0^i  =0$ 
$\left(\displaystyle \sum_{i=0}^{m}A_i  \lambda_0^i y =0\right)$.	
\end{definition}

Similarly, one can define a left quaternion matrix polynomial and its left 
eigenvalues as follows.

\begin{definition}
An $n \times n$ left quaternion matrix polynomial of degree $m$ is a map 
$Q: \mathbb{H} \rightarrow M_n(\mathbb{H})$ given by $Q(\lambda)= \displaystyle \sum_{i=0}^{m} \lambda^i A_i$, with $A_i \in M_n(\mathbb{H})$ and $A_m \neq 0$ 
(the indeterminate $\lambda$ is on the left of the matrix coefficients). A quaternion $\lambda_0$ is a left eigenvalue of $Q(\lambda)$ if there exists a  nonzero vector 
$y \in \mathbb{H}^n$ such that $\displaystyle \sum_{i=0}^{m} \lambda_0^i A_i y =0$. 	
\end{definition}

However, defining right eigenvalues for left quaternion matrix polynomials is not possible 
in the same manner. The exact number of left eigenvalues for a quaternion matrix 
polynomial remains undetermined, making this an underexplored area in mathematical literature. As a result, our focus in this manuscript is entirely on right eigenvalues of right quaternion matrix polynomials, which we henceforth refer to as quaternion matrix polynomials. We also refer to the right eigenvalues as eigenvalues, to avoid confusion. Additional notations and definitions will be introduced as and when needed later on.

\medskip

In the following two examples, we illustrate the difficulty in finding the 
eigenvalues of quaternion matrices and quaternion matrix polynomials due to 
noncommutativity of quaternion multiplication. The first one is an example of a quaternion matrix.

\begin{example}
Let $A = \begin{bmatrix}
	0 & j \\
	i & 0
\end{bmatrix}$. If  $\lambda_0 \in \mathbb{H}$ is a right eigenvalue of $A$, 
then there is a nonzero vector 
$x = \begin{bmatrix}
	x_1 \\
	x_2
\end{bmatrix} \in \mathbb{H}^2$ such that $Ax = x\lambda_0$. On expanding, we get 
$\begin{bmatrix}
	0 & j \\
	i & 0
\end{bmatrix} \begin{bmatrix}
x_1 \\
x_2
\end{bmatrix} = \begin{bmatrix}
x_1 \\
x_2
\end{bmatrix} \lambda_0$. This implies
$jx_2 = x_1 \lambda_0$, $ix_1 = x_2 \lambda_0$ and both $x_1$, $x_2$ are nonzero quaternions. On simplification, we get $kx_2 = x_2 \lambda_0^2$. Due to noncommutativity 
of quaternions, solving this equation is challenging. Similar challenges occur in finding left eigenvalues as well.
\end{example}

The following is an example of a quaternion matrix polynomial where computing the eigenvalues is not 
obvious. We skip the calculations.
\begin{example}
Consider $P(\lambda)= \begin{bmatrix}
1 & 0 \\
0 & 0
\end{bmatrix} \lambda^2 + \begin{bmatrix}
0 & 0 \\
i & 0
\end{bmatrix} \lambda + \begin{bmatrix}
j & 0 \\
0 & 1
\end{bmatrix}$. 
\end{example}

\section[Description]{A brief discussion of the results obtained}\label{sec-3}
A brief description of the results obtained in this manuscript follows. We define the notions 
of stability and hyperstability for quaternion matrix polynomials. Due to  noncommutativity 
of quaternion multiplication, these definitions are not the same as in the complex case. We 
begin by proving some results for quaternion matrix polynomials that are similar to those 
that appear in \cite{Oskar-Wojtylak} for complex matrix polynomials. In particular, 
we prove that if a matrix polynomial is hyperstable with respect to a subset $\Omega$ of $\mathbb{H}$, then it is stable with respect to $\Omega$. Since multiplication in 
quaternions is noncommutative, the proofs are not straightforward generalizations from 
the complex case, and hence we give detailed proofs 
(see Propositions \ref{Prop-stability and eigenvalues} and \ref{Prop-implications}).    
An important result (Theorem \ref{Thm-relation bw matrix polynomial and its complex adjoint}) of this manuscript gives a relation between (hyper)stability of a quaternion matrix polynomial and (hyper)stability of the corresponding complex adjoint matrix polynomial 
(see Definition \ref{complex-adjoint matrix polynomial}). We use this result to prove that stability and hyperstability are generally not equivalent for quaternion matrix polynomials. We then 
proceed to prove that in order to verify stability of a quaternion matrix polynomial with 
respect to an open or closed ball in the set of quaternions, centered at any complex number, 
it suffices to verify its stability with respect to a smaller set in $\mathbb{C}$ (Theorem \ref{Thm-stability on ball}). As an application of Theorems \ref{Thm-relation bw matrix polynomial and its complex adjoint} and \ref{Thm-stability on ball}, we prove stability of a quaternion matrix polynomial with respect to two concentric balls centered at the origin 
(Theorem \ref{H-T thm}). This can be viewed as an analogue 
of a similar result obtained by Higham and Tisseur for complex matrix polynomials 
(Lemma $3.1$ of \cite{Higham-Tisseur}). A generalization of the 
Enestr{\"o}m-Kakeya theorem to quaternion matrix polynomials (Theorem 
\ref{E-K thm Q}) is yet another interesting application of the results obtained in this 
section. Some classes of quaternion matrix polynomials for which stability 
and hyperstability are equivalent are also brought out. We end with a result 
on verifying hyperstability of a quaternion matrix polynomial by introducing these 
notions for certain specific multivariate quaternion matrix polynomials.

\section{Main results}\label{sec-4}

The main results are presented in this section.

\subsection{Stability and hyperstability of quaternion matrix polynomials in one 
variable} 
\hspace*{\fill}\label{subsec-4.1}

We deal with single variable quaternion matrix polynomials and prove results that are 
analogous to those from \cite{Oskar-Wojtylak}. We begin with the definitions of stability, hyperstability and numerical range of quaternion matrix polynomials. Since multiplication 
in $\mathbb{H}$ is noncommutative, these definitions are not the same 
as in the complex case.  Let $\Omega$ denote any nonempty subset of the set of 
quaternions $\mathbb{H}$ and 
\begin{equation}\label{eq-matrix polynomial}
P(\lambda) = A_m \lambda^m + A_{m-1} \lambda^{m-1} + \cdots + A_1 \lambda +A_0 
\end{equation} be an $n \times n$ matrix polynomial of degree $m$, where 
$A_0, A_1, \ldots, A_m \in M_n(\mathbb{H})$ and $A_m \neq 0$.  

\medskip
\begin{definition}\label{Def-stability-quaternion}
$P(\lambda)$ is said to be stable with respect to $\Omega$ if for 
any nonzero vector $y \in \mathbb{H}^n$ and for any $\mu \in \Omega$ there exists a 
nonzero vector $z \in \mathbb{H}^n$ such that 
\begin{equation}
z^*A_m y \mu^m + z^*A_{m-1} y \mu^{m-1} + \cdots +z^*A_1 y \mu +z^*A_0 y \neq 0.
\end{equation}
\end{definition}

\medskip
\begin{definition}\label{Def-hyperstability-quaternion}
$P(\lambda)$ is said to be hyperstable with respect to $\Omega$ if for any nonzero vector 
$y \in \mathbb{H}^n$ there exists a nonzero vector $z \in \mathbb{H}^n$ 
such that 
\begin{equation}
z^*A_m y \mu^m + z^*A_{m-1} y \mu^{m-1} + \cdots +z^*A_1 y \mu +z^*A_0 y \neq 0 \ \ \text{for all}\, \mu \in \Omega.
\end{equation}
\end{definition}

\medskip
\begin{definition}\label{Def-numerical range}
The numerical range of $P(\lambda)$ as given in \eqref{eq-matrix polynomial} is the set 
\begin{equation*}
W(P) = \{ \lambda \in \mathbb{H} \ : \ y^*A_m y \lambda^m + \cdots + 
y^* A_1 y \lambda + y^* A_0 y = 0 \ \text{for some nonzero $y \in \mathbb{H}^n$}\}.
\end{equation*}
\end{definition}

\noindent
The following result gives a relation between the location of eigenvalues and the above 
notion of stability with respect to any subset of $\mathbb{H}$. 

\medskip
\begin{proposition}\label{Prop-stability and eigenvalues}
Let $P(\lambda)$ be as in \eqref{eq-matrix polynomial} and $\Omega \subseteq 
\mathbb{H}$ be nonempty. Then $P(\lambda)$ has no eigenvalues in $\Omega$ 
if and only if it is stable with respect to $\Omega$.
\end{proposition}

\begin{proof}
Assume that $P(\lambda)$ has no eigenvalues in $\Omega$. If $P(\lambda)$ is not 
stable with respect to $\Omega$, then there exists a nonzero vector 
$y \in \mathbb{H}^n$ and a quaternion $\mu \in \Omega$ such that for any 
$z \in \mathbb{H}^n \setminus \{0\}$, 
\begin{equation}
z^*A_m y \mu^m + z^*A_{m-1} y \mu^{m-1} + \cdots +z^*A_1 y \mu +z^*A_0 y = 0. 
\end{equation} 
This implies 
\begin{equation}
z^* \left(A_m y \mu^m +A_{m-1} y \mu^{m-1} + \cdots +A_1 y \mu +A_0 y \right) = 0. 
\end{equation}
Taking $z=e_i$, where $e_i$ is the vector in $\mathbb{H}^n$ with $1$ in the $i$th 
component and zeros elsewhere, we get 
\begin{equation}
A_m y \mu^m +A_{m-1} y \mu^{m-1} + \cdots +A_1 y \mu +A_0 y =0, 
\end{equation} contradicting the assumption that $P(\lambda)$ has no eigenvalue 
in $\Omega$.\\
Conversely, let $P(\lambda)$ be stable with respect to $\Omega$. 
If $\mu \in \Omega$ is an eigenvalue of $P(\lambda)$, then there exists a nonzero vector 
$y \in \mathbb{H}^n$ such that 
\begin{equation}
A_m y \mu^m +A_{m-1} y \mu^{m-1} + \cdots +A_1 y \mu +A_0 y = 0. 
\end{equation}
Then for any nonzero vector $z \in \mathbb{H}^n$, 
\begin{equation}
z^*A_m y \mu^m + z^*A_{m-1} y \mu^{m-1} + \cdots +z^*A_1 y \mu +z^*A_0 y = 0, 
\end{equation}
which contradicts the assumption that $P(\lambda)$ is stable with respect to $\Omega$. 
This completes the proof.
\end{proof}

\noindent
We now prove that hyperstability of a matrix polynomial is stronger than stability with 
respect to $\Omega$, even when its numerical range does not intersect $\Omega$.

\medskip
\begin{proposition}\label{Prop-implications}
Let $P(\lambda)$ be as in \eqref{eq-matrix polynomial} and let $\Omega \subseteq \mathbb{H}$ 
be nonempty. Consider the following conditions:
\begin{itemize}
\item[(a)]  $\Omega \cap W(P) = \emptyset$. 
\item[(b)] $P(\lambda)$ is hyperstable with respect to $\Omega$.
\item[(c)] $P(\lambda)$ is stable with respect to $\Omega$.
\end{itemize}
Then \emph{(a) $\implies$ (b) $\implies$ (c)}.  
\end{proposition}

\begin{proof}
We first prove (a) $\implies$ (b). Let $y \in \mathbb{H}^n \setminus \{0\}$. Since 
$W(P) \cap \Omega = \emptyset$, for any $\lambda \in \Omega$, we have 
$\displaystyle \sum_{i=0}^{m} y^*A_i y \lambda^i \neq 0$. In particular, taking 
$z =y$, we obtain $\displaystyle \sum_{i=0}^{m} z^*A_i y \lambda^i \neq 0$ for all 
$\lambda \in \Omega$. This proves that $P(\lambda)$ is hyperstable with respect to $\Omega$. 
The second implication follows from the definition of stability.     
\end{proof}

\smallskip
\noindent
As in  the complex case, the reverse implications are not true in Proposition 
\ref{Prop-implications}. Our first major result of this 
manuscript (Theorem \ref{Thm-relation bw matrix polynomial and its complex adjoint}), 
which gives a relation between (hyper)stability of a quaternion matrix polynomial and its 
complex adjoint matrix polynomial illustrates this. We begin with the 
definition of the complex adjoint matrix of a quaternion matrix.

\begin{definition}\label{complex-adjoint matrix}
 Any matrix 
$A \in M_n(\mathbb{H})$ can be expressed as $A=A_1 +A_2 j$, where 
$A_1$, $A_2 \in M_n(\mathbb{C})$. The complex adjoint matrix of $A$ is a 
$2n \times 2n$ complex block matrix defined as 
$\bigchi_A := \begin{bmatrix}
	A_1 & A_2 \\
	-\bar{A_2} & \bar{A_1}.
\end{bmatrix}$.	
\end{definition}

Readers may refer to  \cite{Rodman} or \cite{Zhang} for more details about various 
properties of the complex adjoint matrix. The above definition suggests that given a 
quaternion matrix polynomial $P(\lambda)$, we can associate to $P(\lambda)$, a 
complex matrix polynomial $P_{\chi}(\lambda)$ as follows.

\begin{definition}\label{complex-adjoint matrix polynomial}
Given an $n \times n$ quaternion matrix polynomial $P(\lambda)=\displaystyle 
\sum_{i=0}^{m} A_i\lambda^i$ we associate a $2n \times 2n$ complex matrix polynomial 
$P_{\chi}: \mathbb{C} \rightarrow M_{2n}(\mathbb{C})$ 
defined by $P_\chi(\lambda)=\displaystyle \sum_{i=0}^{m} \bigchi_{A_i}\lambda^i$. 
We call $P_\chi(\lambda)$ as the complex adjoint matrix polynomial of $P(\lambda)$. 
\end{definition}

\begin{theorem}\label{Thm-relation bw matrix polynomial and its complex adjoint}
Let $P(\lambda)$ be as in \eqref{eq-matrix polynomial} and 
$\Omega \subseteq \mathbb{H}$ be nonempty. Then (i) $P(\lambda)$ is stable with 
respect to $\Omega \cap \mathbb{C}$ if and only if its complex adjoint matrix polynomial 
$P_\chi (\lambda)$ is stable with respect to $\Omega \cap \mathbb{C}$. 
(ii) $P(\lambda)$ is hyperstable with respect to $\Omega \cap \mathbb{C}$ 
if and only if its complex adjoint matrix polynomial $P_\chi (\lambda)$ is hyperstable with respect to $\Omega \cap \mathbb{C}$.
\end{theorem}

\begin{proof}
We only prove the first statement as the proof of the second statement follows similarly. 
(i) Let $P(\lambda)$ be stable with respect to $\Omega \cap \mathbb{C}$. 
To prove that $P_{\chi}(\lambda)$ is stable with respect to $\Omega \cap \mathbb{C}$, consider a nonzero vector 
$Y=\begin{bmatrix}
	y_1 \\
	y_2
\end{bmatrix} \in \mathbb{C}^{2n}$ and $\mu \in \Omega \cap \mathbb{C}$, where $y_1$, 
$y_2 \in \mathbb{C}^n$. Define $y := y_1 -\overline{y}_2j$. Then 
$y \in \mathbb{H}^n \setminus \{0\}$. Since $P(\lambda)$ is stable with respect to 
$\Omega \cap \mathbb{C}$, there exists $z \in \mathbb{H}^n \setminus \{0\}$ such that 
$\displaystyle \sum_{i=0}^{m} z^*A_i y \mu^i \neq 0$. Writing $z = z_1 +z_2 j$, 
where $z_1$, $z_2 \in \mathbb{C}^n$ and $A_i = A_{i1} +A_{i2}$, where $A_{i1}$, $A_{i2} 
\in M_n(\mathbb{C})$ for $1 \leq i \leq m$, we have $z^* = z_1^*-z_2^Tj$ and 
\begin{equation*}
\displaystyle \sum_{i=0}^{m} z^*A_i y   \mu^i = \sum_{i=0}^{m}(z_1^* -z_2^Tj) (A_{i1}+A_{i2}j) 
(y_1 -\overline{y}_2j) \mu^i \neq 0.  
\end{equation*} 

\noindent
On expanding, we get

$\displaystyle \sum_{i=0}^{m} \left( z_1^*A_{i1}y_1 +z_1^*A_{i2}y_2 +z_2^T \overline{A}_{i2}y_1 
- z_2^T \overline{A}_{i1} y_2 \right) \mu^i + \\
\left(\displaystyle \sum_{i=0}^{m} \left( z_1^*A_{i2}\overline{y}_1 -z_1^*A_{i1} \overline{y}_2 
-z_2^T\overline{A}_{i1}\overline{y}_1 - z_2^T \overline{A}_{i2} \overline{y}_2 \right) 
\overline{\mu}^i \right)j \neq 0$. 

\noindent
We therefore have 
\begin{equation} \label{Eqn-3.9}
\displaystyle \sum_{i=0}^{m} \left( z_1^*A_{i1}y_1 +z_1^*A_{i2}y_2 + z_2^T \overline{A}_{i2}y_1 
- z_2^T \overline{A}_{i1}y_2 \right) \mu^i \neq 0 
\end{equation} or
\begin{equation}\label{Eqn-3.10}
\displaystyle \sum_{i=0}^{m} \left( z_1^*A_{i2}\overline{y}_1 -z_1^*A_{i1} \overline{y}_2 
-z_2^T\overline{A}_{i1}\overline{y}_1 - z_2^T \overline{A}_{i2} \overline{y}_2 \right) \overline{\mu}^i \neq 0. 
\end{equation} 

\noindent
From Equations \eqref{Eqn-3.9} and \eqref{Eqn-3.10} we have 
\begin{equation}
\sum_{i=0}^{m} \begin{bmatrix}
z_1^* & -z_2^T
\end{bmatrix} \begin{bmatrix}
A_{i1} & A_{i2} \\
-\overline{A}_{i2} & \overline{A}_{i1}
\end{bmatrix} \begin{bmatrix}
y_1 \\
y_2
\end{bmatrix} \mu^i \neq 0,
\end{equation} or 
\begin{equation}
\sum_{i=0}^{m} \begin{bmatrix}
-z_2^* & -z_1^T
\end{bmatrix} \begin{bmatrix}
A_{i1} & A_{i2} \\
-\overline{A}_{i2} & \overline{A}_{i1}
\end{bmatrix} \begin{bmatrix}
y_1 \\
y_2
\end{bmatrix} \mu^i \neq 0.
\end{equation} 

\noindent
This implies 
$\begin{bmatrix}
		z_1 & -\overline{z_2}
\end{bmatrix}^* P_{\chi}(\mu) \begin{bmatrix}
y_1 \\
y_2
\end{bmatrix} \neq 0$ or $\begin{bmatrix}
-z_2 & -\overline{z_1}
\end{bmatrix}^* P_{\chi}(\mu) \begin{bmatrix}
y_1 \\
y_2
\end{bmatrix} \neq 0$.
This proves that $P_\chi(\lambda)$ is stable with respect to $\Omega \cap \mathbb{C}$. 
\\ Conversely, assume that $P_\chi(\lambda)$ is stable with respect to 
$\Omega \cap \mathbb{C}$. To prove that $P(\lambda)$ is stable with 
respect to $\Omega \cap \mathbb{C}$, consider $y \in \mathbb{H}^n \setminus \{0\}$ 
and $\mu \in \Omega \cap \mathbb{C}$. Write $y = y_1 + y_2 j$, where 
$y_1$, $y_2 \in \mathbb{C}^n$. 
Define $Y := \begin{bmatrix}
y_1 \\
-y_2
\end{bmatrix} \in \mathbb{C}^{2n} \setminus \{0\}$. Since $P_{\chi}(\lambda)$ 
is stable with respect to $\Omega \cap \mathbb{C}$, there exists 
$Z \in \mathbb{C}^{2n} \setminus \{0\}$ such that $Z^*P_\chi(\mu)Y \neq 0$. If 
$Z = \begin{bmatrix}
z_1 \\
z_2
\end{bmatrix}$, where $z_1$, $z_2 \in \mathbb{C}^n$, then 
\begin{equation*}
Z^*P_\chi(\mu)Y =\displaystyle \sum_{i=0}^{m} \begin{bmatrix}
z_1^* & z_2^*
\end{bmatrix} \begin{bmatrix}
A_{i1} & A_{i2} \\
-\overline{A}_{i2} & \overline{A}_{i1}
\end{bmatrix} \begin{bmatrix}
y_1 \\
-\overline{y}_2
\end{bmatrix} \mu^i \neq 0.
\end{equation*} 

\noindent
On multiplication we get 
\begin{equation}\label{Eqn-3.13}
\displaystyle \sum_{i=0}^{m} \left( z_1^*A_{i1}y_1 - z_1^*A_{i2}\overline{y}_2 -z_2^* 
\overline{A}_{i2}y_1 -z_2^*\overline{A}_{i1}\overline{y}_2 \right) \mu^i \neq 0.
\end{equation} 

\noindent
Define $z := z_1 -\overline{z}_2j \in \mathbb{H}^n \setminus \{0\}$.  Then, 
\begin{align}
\sum_{i=0}^{m} z^* A_i y \mu^i & = 
\sum_{i=0}^{m} (z_1^* + z_2^* j) (A_{i1} + A_{i2}j) (y_1 + y_2 j) \mu^i \notag \\
&= \sum_{i=0}^{m} \left( z_1^*A_{i1}y_1 - z_1^*A_{i2}\overline{y}_2 -z_2^* \overline{A}_{i2}y_1 
-z_2^*\overline{A}_{i1}\overline{y}_2 \right) \mu^i \notag \\
&\quad + 
\left( \sum_{i=0}^{m} \left( z_1^*A_{i1}y_2 + z_1^*A_{i2}\overline{y}_1 -z_2^* 
\overline{A}_{i2}y_2 +z_2^*\overline{A}_{i1}\overline{y}_1 \right) 
\overline{\mu}^i \right) j \label{Eqn-3.14}.
\end{align}

\noindent
From Equations \eqref{Eqn-3.13} and \eqref{Eqn-3.14}, we get 
$\displaystyle \sum_{i=0}^{m} z^* A_i y \mu^i \neq 0$. This proves that $P(\lambda)$ 
is stable with respect to $\Omega \cap \mathbb{C}$. 
\end{proof}

\medskip
\noindent
The following examples illustrate that the reverse implication is not true in Proposition 
\ref{Prop-implications}.

\medskip
\begin{example}  \label{ex-1}
The numerical range of a matrix $A$ and the numerical range of a linear matrix 
polynomial $P(\lambda) = I\lambda - A$ coincide. Therefore, condition (a) may not 
necessarily be equivalent to conditions (b) and (c). This is supported by the following 
example. Let $P(\lambda) =I\lambda-A$, where 
$A = \begin{bmatrix}
1 & 0 \\
0 & 0
\end{bmatrix}$. It is easy to verify that the numerical range of $P(\lambda)$ is the 
interval $[0,1]$ on the real line. Note that from Theorem $5.4$ of \cite{Zhang}, the only 
right eigenvalues of $P(\lambda)$ are $0$ and $1$. Therefore $P(\lambda)$ is stable 
with respect to $\Omega =\mathbb{H} \setminus \{0,1\}$. We now show that 
$P(\lambda)$ is also hyperstable with respect to $\Omega$. Let 
$y = \begin{bmatrix}
y_1 \\
y_2
\end{bmatrix} \in \mathbb{H}^2 \setminus \{0\}$. For $y_2 \neq 0$, consider 
$z = \begin{bmatrix}
0 \\
1
\end{bmatrix}$. Then for any $\mu \in \Omega$,  

\begin{equation}
\begin{bmatrix}
0 & 1 
\end{bmatrix} \begin{bmatrix}
1 & 0\\
0 & 1
\end{bmatrix} \begin{bmatrix}
y_1 \\
y_2
\end{bmatrix} \mu - \begin{bmatrix}
0 & 1 
\end{bmatrix} \begin{bmatrix}
1 & 0\\
0 & 0
\end{bmatrix} \begin{bmatrix}
y_1 \\
y_2
\end{bmatrix}= y_2 \mu \neq 0. 
\end{equation} 

\noindent
If $y_2=0$, then $y_1 \neq 0$. In this case, take 
$z = \begin{bmatrix}
1 \\
0
\end{bmatrix}$. Then for any $\mu \in \Omega$,
\begin{equation}
\begin{bmatrix}
1 & 0 
\end{bmatrix} \begin{bmatrix}
1 & 0\\
0 & 1
\end{bmatrix} \begin{bmatrix}
y_1 \\
y_2
\end{bmatrix} \mu - \begin{bmatrix}
1 & 0 
\end{bmatrix} \begin{bmatrix}
1 & 0\\
0 & 0
\end{bmatrix} \begin{bmatrix}
y_1 \\
y_2
\end{bmatrix}= y_1 (\mu +1) \mu \neq 0.  
\end{equation} 

\noindent
Thus $P(\lambda)$ is hyperstable with respect to $\Omega$.  Note that 
$W(P) \cap \Omega \neq \emptyset$.
\end{example}

\medskip
\begin{example} \label{ex-2}
This example shows that stability need not imply hyperstability. 
Let $P\colon \mathbb{H} \rightarrow M_2(\mathbb{H})$ be defined by 
$P(\lambda) = \begin{bmatrix}
1 & \lambda \\
\lambda & \lambda^2+1
\end{bmatrix}$. 
Let $y = \begin{bmatrix}
y_1 \\
y_2
\end{bmatrix} \in \mathbb{H}^2$. If  
$ \begin{bmatrix}
0 & 0 \\
0 & 1
\end{bmatrix} \begin{bmatrix}
y_1 \\
y_2
\end{bmatrix} \lambda^2 +
\begin{bmatrix}
0 & 1 \\
1 & 0
\end{bmatrix} \begin{bmatrix}
y_1 \\
y_2
\end{bmatrix} \lambda + 
\begin{bmatrix}
1 & 0 \\
0 & 1
\end{bmatrix} \begin{bmatrix}
y_1 \\
y_2
\end{bmatrix} = \begin{bmatrix}
y_2 \lambda + y_1 \\
y_2 \lambda^2 + y_1 \lambda + y_2
\end{bmatrix} = \begin{bmatrix}
0 \\
0
\end{bmatrix}$, then $y_2 \lambda + y_1  =0$ and 
$y_2 \lambda^2 + y_1 \lambda +y_2=0$. This gives $y_1=0=y_2$ and so $y$ is the 
zero vector. Therefore, no quaternion is an eigenvalue of $P(\lambda)$ and it is stable 
with respect to any subset $\Omega \subseteq \mathbb{H}$. 
Let $\Omega = \{q \in \mathbb{H}: |q| \leq 1\}$ and 
$\Omega^{\prime} = \{ c \in \mathbb{C}$ : $|c| \leq 1 \}$. Note that 
$\Omega \cap \mathbb{C} =\Omega^{\prime}$. We show that the complex adjoint 
matrix polynomial $P_\chi(\lambda)$ is not hyperstable with respect to 
$\Omega^{\prime}$. Note that 
$P_\chi(\lambda) = \begin{bmatrix}
0 & 0 & 0 & 0 \\
0 & 1 & 0 & 0 \\
0 & 0 & 0 & 0 \\
0 & 0 & 0 & 1 
\end{bmatrix}\lambda^2 + \begin{bmatrix}
0 & 1 & 0 & 0 \\
1 & 0 & 0 & 0 \\
0 & 0 & 0 & 1 \\
0 & 0 & 1 & 0 
\end{bmatrix}\lambda + \begin{bmatrix}
1 & 0 & 0 & 0 \\
0 & 1 & 0 & 0 \\
0 & 0 & 1 & 0 \\
0 & 0 & 0 & 1 
\end{bmatrix}$. If $y = \begin{bmatrix}
0 & 0 & 0 & 1
\end{bmatrix}^T$, then for any nonzero vector $z =\begin{bmatrix}
z_1 &z_2 & z_3 & z_4
\end{bmatrix}^T \in \mathbb{C}^4$, $z^*P_\chi(\lambda)y$ can be computed as  $\overline{z}_3\lambda 
+\overline{z}_4 (\lambda^2 +1)$. We then have the following 
cases:\\
(i) If $z_3 =z_4 =0$, then for any $\lambda_0 \in \Omega^{\prime}$, 
$z^*P_\chi(\lambda_0)y=0$. \\
(ii) If $z_3 \neq 0$, $z_4 =0$, then for $\lambda_0=0 \in \Omega'$, 
$z^*P_\chi(\lambda_0)y=0$.\\
(iii) If $z_3 = 0$, $z_4 \neq 0$, then for $\lambda_0=i \in \Omega'$, 
$z^*P_\chi(\lambda_0)y=0$. \\
(iv) If $z_3 \neq 0$, $z_4 \neq 0$, then $z^*P_\chi(\lambda)y= 
\overline{z}_3\lambda +\overline{z}_4(\lambda^2 +1) = \overline{z}_4 \lambda^2+ 
\overline{z}_3 \lambda + \overline{z}_4$, a quadratic polynomial with complex 
coefficients. Since the product of its roots is equal to 
$\displaystyle \frac{\overline{z}_4}{\overline{z}_4} =1$, 
one of the roots has to lie inside $\Omega^{\prime}$. Therefore, with respect to $\Omega^{\prime}$, $P_{\chi}(\lambda)$ is not hyperstable. From Theorem 
\ref{Thm-relation bw matrix polynomial and its complex adjoint}, $P(\lambda)$ is not hyperstable with respect to $\Omega^{\prime},$ which in turn implies that $P(\lambda)$ 
is not hyperstable with respect to $\Omega$. 
\end{example}

\medskip
\noindent
In the result that follows, we simplify stability analysis of a quaternion matrix polynomial 
by reducing verification of quaternion stability to a subset of complex numbers. 
In particular we prove that in order to check for stability of a quaternion matrix 
polynomial with respect to an open (or closed) ball in the set of quaternions, that is 
centered at any complex number, it is sufficient to check its stability with respect 
to a smaller set contained in the set of complex numbers. We begin with a few notations. 
Let $B(p;r)= \{q \in \mathbb{H} : |q-p|<r \}$ represent the open ball in $\mathbb{H}$ centered at $p \in \mathbb{H}$ and radius $r >0$ and let 
$\overline{B(p;r)}= \{q \in \mathbb{H} : |q-p| \leq r \}$ be its closure. We prove this 
result for open balls and the proof follows verbatim for 
closed balls. We make use of the following lemma.

\medskip
\begin{lemma}\label{Lem-stability over ball}
Let $P(\lambda)$ be as in \eqref{eq-matrix polynomial} and $B\left(a; r \right)$ be an 
open ball in $\mathbb{H}$ centered at $a \in \mathbb{C}$ of radius $r >0$. Then 
$P(\lambda)$ is stable with respect to $B (a; r) \cap \mathbb{C}$ if and only if 
$P(\lambda)$ is stable with respect to $B (\overline{a}; r) \cap \mathbb{C}$. 
\end{lemma}

\begin{proof}
Suppose $P(\lambda)$ is stable with respect to $B (a; r) \cap \mathbb{C}$. 
Take a nonzero vector $y \in \mathbb{H}^n$ and let $\mu \in B (\overline{a}; r) \cap \mathbb{C}$ so 
that $|\mu -\overline{a}| <r$ . Since $|\overline{\mu} -a|= |\mu -\overline{a}|$ we have $\overline{\mu} 
\in B (a; r) \cap \mathbb{C}$. Therefore,  for 
the nonzero vector $w: = -yj \in \mathbb{H}^n$ and 
$\overline{\mu} \in B \left(a; r \right) \cap \mathbb{C},$ there exists a nonzero vector 
$z \in \mathbb{H}^n$ such that
$z^*A_mw \overline{\mu}^m + \cdots + z^*A_1w \overline{\mu} +z^*A_0w= z^*A_m(-yj)\overline{\mu}^m  
+ \cdots + z^*A_1(-yj) \overline{\mu} +z^*A_0(-yj) \neq 0$. 
Post-multiplying the above identity by $j$ we have
\begin{equation*}
z^*A_m(-yj)\overline{\mu}^mj  + \cdots + z^*A_1(-yj) \overline{\mu}j + 
z^*A_0(-yj)j \neq 0. 
\end{equation*}
Since $jq = \overline{q}j$ and $j^2 =-1$ we get 
\begin{equation*}
z^*A_my\mu^m  + \cdots + z^*A_1y \mu +z^*A_0y \neq 0.
\end{equation*}
This proves that $P(\lambda)$ is stable with respect to 
$B\left(\overline{a}; r \right) \cap \mathbb{C}$. The converse follows from a similar argument.
\end{proof}

\noindent
We are now in a position to prove the aforesaid theorem.

\medskip
\begin{theorem}\label{Thm-stability on ball}
Let $P(\lambda)$ be as in \eqref{eq-matrix polynomial} and $\Omega = B (a; r)$ be an 
open ball in $\mathbb{H}$ centered at $a \in \mathbb{C}$ of radius $r >0$. Then 
$P(\lambda)$ is stable with respect to $\Omega$ if and only if  $P(\lambda)$ is stable 
with respect to $\Omega \cap \mathbb{C}$.
\end{theorem}

\begin{proof}
Assume that $P(\lambda)$ is stable with respect to $\Omega \cap \mathbb{C}$. 
Suppose $P(\lambda)$ is not stable with respect to 
$\Omega$. We first consider the case when $a$ is a complex number with 
nonnegative imaginary part. Then $P(\lambda)$ has an eigenvalue 
$\mu \in \Omega$ which is not a complex number. Since every eigenvalue 
lies in one of the equivalence classes of standard eigenvalues, $\mu$ is similar to a 
standard eigenvalue $\mu_0 \in \mathbb{C}$ of $P(\lambda)$. Let $q$ be a nonzero quaternion such that $\mu = q^{-1}\mu_0 q$. 
By Lemma $3.3$ of \cite{Ahmad-Ali-Ivan}, we have 
$|\mu_0 -a| \leq |q^{-1}\mu_0 q - a |= |\mu -a|$. Since $|\mu -a| < r$, 
we get $|\mu_0 -a| <r$, a contradiction to the fact that $P(\lambda)$ is stable with
respect to $\Omega \cap \mathbb{C}$. This proves that $P(\lambda)$ is stable with 
respect to $ \Omega$ and the first step is complete. 

If $a$ is a complex number whose imaginary part is negative, then $\overline{a}$ 
is a complex number with positive imaginary part. By Theorem 
$2.1 (5)$ of \cite{Zhang}, $a$ and $\overline{a}$ are similar, and hence 
$p^{-1}\overline{a}p = a$ for some $p \in \mathbb{H} \setminus \{0\}$. Once again, 
by Lemma $3.3$ of \cite{Ahmad-Ali-Ivan}, we have 
\begin{equation}
|\mu_0 -\overline{a}| \leq |q^{-1}\mu_0 q - p^{-1} \overline{a} p| = |\mu - a| <r.  
\end{equation} 
We thus get $\mu_0 \in B \left(\overline{a}; r \right) \cap \mathbb{C}$. This shows 
that $P(\lambda)$ is not stable with respect to 
$B \left(\overline{a}; r \right) \cap \mathbb{C}$.  From 
Lemma \ref{Lem-stability over ball}, we conclude that $P(\lambda)$ is not stable with 
respect to $\Omega \cap \mathbb{C}$ as well. Therefore for any $a \in \mathbb{C}$, $P(\lambda)$ 
is stable with respect to $B \left(a; r \right)$. For the converse, it is clear 
that if $P(\lambda)$ is stable with respect to $\Omega$, then it is stable with respect to 
$\Omega \cap \mathbb{C}$, thereby completing the proof.
\end{proof}

\medskip
\noindent
Some remarks are in order.

\medskip
\begin{remark}\label{rem-1}
Theorem \ref{Thm-stability on ball} is not true if 
$a \in \mathbb{H} \setminus \mathbb{C}$, as the following example illustrates. Let $P(\lambda) = \begin{bmatrix}
1 & 0 \\
0 & 1
\end{bmatrix} \lambda + \begin{bmatrix}
j & 0 \\
0 & j
\end{bmatrix}$. It is easy to verify that the set of eigenvalues of $P(\lambda)$ is 
$\sigma(P) = \{ \mu \in \mathbb{H} : \mu =s^{-1}js \ $\text{for some nonzero 
$s \in \mathbb{H}$}$ \}$. Let $\Omega = B \left(j;1 \right)$. Notice that $P(\lambda)$ 
is not stable with respect to $\Omega$ as $j \in \Omega$. From Theorem $2.2$ of 
\cite{Zhang}, we have $\sigma(P) \cap \mathbb{C} = \{-i, i\}$. Therefore, $-i$ and $i$ 
are the only complex eigenvalues of $P(\lambda)$. Since $-i, i \notin \Omega \cap \mathbb{C}$, $P(\lambda)$ is stable with respect to $\Omega \cap \mathbb{C}$. 
This example also suggests that Theorem \ref{Thm-stability on ball} does not hold for 
an arbitrary convex set. 
\end{remark}

\subsection{Stable sets for arbitrary quaternion matrix polynomials}\hspace*{\fill}\label{subsec-4.2}

In this section, we make use of Theorems \ref{Thm-relation bw matrix polynomial and 
its complex adjoint} and \ref{Thm-stability on ball} to obtain particular sets 
with respect to which quaternion matrix polynomials are stable. The main results 
(Theorems \ref{Thm-lower bound for eigenvalue}, \ref{Thm-upper bound for eigenvalue} 
and \ref{H-T thm}) are analogous to Lemma $3.1$ of \cite{Higham-Tisseur}, due to 
Higham and Tisseur for complex matrix polynomials. We first prove that a quaternion 
matrix polynomial is stable with respect to a particular open ball centered at zero.
For $A \in M_n(\mathbb{H})$, the spectral norm is denoted by $||A||$ and is defined 
as $||A|| = \displaystyle \sup_{x \neq 0} \Bigg\{\frac{||Ax||_2}{||x||_2}: x \in 
\mathbb{H}^n \Bigg\}$. Note that $||\bigchi_A|| = ||A||$.

\medskip
\begin{theorem}\label{Thm-lower bound for eigenvalue}
Let $P(\lambda)$ be as in \eqref{eq-matrix polynomial} with $A_0$ invertible. 
Let $l(z) = ||A_m|| z^m + \cdots + ||A_1||z - ||A_0^{-1}||^{-1}$ be a complex polynomial. 
Then $P(\lambda)$ is stable with respect to 
$\Omega = B \left(0; r \right) \subseteq \mathbb{H}$, where $r$ is the unique positive 
real zero of $l(z)$. 
\end{theorem}

\begin{proof}
Consider the complex adjoint matrix polynomial of $P(\lambda)$ given by 
$P_{\chi} (\lambda) = \bigchi_{A_m} \lambda^m + \cdots + \bigchi_{A_1} \lambda + \bigchi_{A_0}$. Since $A_0$ is invertible, it follows from Theorem $4.2(5)$ 
of \cite{Zhang} that $\bigchi_{A_0}$ is invertible and 
$\bigchi_{A_0}^{-1} = \bigchi_{A_0^{-1}}$. Therefore, by Lemma $3.1$ of \cite{Higham-Tisseur}, $P_{\chi} (\lambda)$ is stable with respect to 
$ D= \{ z \in \mathbb{C} : |z| < r \}$, where $r$ is the unique positive real root of the 
complex polynomial 
$||\bigchi_{A_m}||z^m + \cdots + ||\bigchi_{A_1}||z - ||\bigchi _ {A_0}^{-1}||^{-1} = ||A_m|| z^m + 
 \cdots + ||A_1||z - ||A_0^{-1}||^{-1} = l(z)$. This proves stability of 
 $P_{\chi}(\lambda)$ with  respect to $\Omega \cap \mathbb{C}$. From Theorem 
\ref{Thm-relation bw matrix polynomial and its complex adjoint}, we conclude that  
$P(\lambda)$ is stable with respect to $\Omega \cap \mathbb{C}$. Finally, 
Theorem \ref{Thm-stability on ball}, yields that $P(\lambda)$ is stable with respect to $\Omega$.
\end{proof}

\medskip
\noindent
We now proceed to prove in Theorem \ref{Thm-upper bound for eigenvalue}, that an 
arbitrary quaternion matrix polynomial is stable with respect to the complement of a 
closed ball centered at zero in $\mathbb{H}$. The following set of lemmas set the 
tone for this. We prove them in the order of preference.

\medskip
\begin{lemma}\label{Lem-zero eigenvalue}
Let $P(\lambda)$ be as in \eqref{eq-matrix polynomial}. Then zero is an eigenvalue of 
$P(\lambda)$ if and only if $A_0$ is non-invertible. 
\end{lemma}

\begin{proof}
Suppose zero is an eigenvalue of $P(\lambda)$ and $y \in \mathbb{H}^n$ is a 
nonzero vector such that $A_m y 0^m + \cdots + A_1 y 0 +A_0 y = 0$. Thus $A_0 y =0$, 
thereby proving that $A_0$ is not invertible. Conversely, if $A_0$ is not invertible, then 
$\bigchi_{A_0}$ is not invertible. Therefore, zero is an eigenvalue of $\bigchi_{A_0}$, 
which in turn implies that zero is an eigenvalue of $A_0$.  We thus have $A_0 y = 0$ 
for some nonzero vector $y \in \mathbb{H}^n$ so that $A_m y 0^m + 
\cdots + A_1 y 0 + A_0y = 0$. This proves that zero is an eigenvalue of $P(\lambda)$.
\end{proof}

\medskip

\begin{lemma}\label{Lem-inverse matrix polynomial}
Let $P(\lambda)$ be as in \eqref{eq-matrix polynomial} with $A_m$ and $A_0$ 
invertible. Then 
$\lambda_0 \in \mathbb{H}$ is an eigenvalue of $P(\lambda)$ if and only if $\frac{1}{\lambda_0}$ is an eigenvalue of 
$Q(\lambda) = P \left(\frac{1}{\lambda} \right)\lambda^m = A_0 \lambda^m + 
A_1 \lambda^{m-1} + \cdots + A_m$.
\end{lemma}

\begin{proof}
Let $\lambda_0 \in \mathbb{H}$ be an eigenvalue of $P(\lambda)$. By Lemma 
\ref{Lem-zero eigenvalue}, $\lambda_0 \neq 0$. Then, 
$A_m y \lambda_0^m +\cdots + A_1 y \lambda_0 + A_0 y =0$ for some 
$y \in \mathbb{H}^n \setminus \{0\}$. Post-multiplying by 
$\frac{1}{\lambda_0^m}$ on both the sides, we get 
$A_0 y \frac{1}{\lambda_0^m} + A_1 y \frac{1}{\lambda_0^{m-1}}+ 
\cdots + A_my = 0$. This implies $\frac{1}{\lambda_0}$ is an eigenvalue of 
$Q(\lambda)$. The converse follows in a similar argument.
\end{proof}

\medskip

\begin{lemma}\label{Lem-stability on complement of ball}
Let $P(\lambda)$ be as in \eqref{eq-matrix polynomial} with $A_m, A_0$ invertible 
and let $\Omega = \overline{B \left(0;r \right)} \subseteq \mathbb{H}$ with $r > 0$. 
If $P(\lambda)$ is stable with respect to 
$\left(\mathbb{H} \setminus \Omega \right) \cap \mathbb{C}$, then $P(\lambda)$ is 
stable with respect to $\mathbb{H} \setminus \Omega$.
\end{lemma}

\begin{proof}
Suppose $P(\lambda)$ is not stable with respect to $\mathbb{H} \setminus \Omega$. 
Then, there exists an eigenvalue $\mu$ of $P(\lambda)$ in  
$\mathbb{H} \setminus \Omega$. By Lemma 
\ref{Lem-inverse matrix polynomial}, $\frac{1}{\mu}$ is an eigenvalue of 
$Q(\lambda) = P\left(\frac{1}{\lambda}\right)\lambda^m$. Then $Q(\lambda)$ is not 
stable with respect to $B \left(0; \frac{1}{r} \right)$. Therefore, by Theorem 
\ref{Thm-stability on ball}, $Q(\lambda)$ is not stable with respect to $B\left(0;\frac{1}{r}\right) 
\cap \mathbb{C}$. This in turn implies there 
exists an eigenvalue $\delta \in \mathbb{C}$ of $Q(\lambda)$, with 
$|\delta| < \displaystyle \frac{1}{r}$. 
Once again by Lemma \ref{Lem-inverse matrix polynomial}, 
$\displaystyle \frac{1}{\delta}$ is an eigenvalue of $P(\lambda)$, thereby implying 
that $P(\lambda)$ is not stable with respect to 
$\left(\mathbb{H} \setminus \Omega \right) \cap \mathbb{C}$. This contradiction 
proves that $P(\lambda)$ is stable with respect to $\mathbb{H} \setminus \Omega$.
\end{proof}

\noindent
We now prove the main result stated at the beginning of this section.

\begin{theorem}\label{Thm-upper bound for eigenvalue}
Let $P(\lambda)$ be as in \eqref{eq-matrix polynomial} with $A_m$ and $A_0$ 
invertible. Let $u(z) =||A_m^{-1}||^{-1} z^m - ||A_{m-1}|| z^{m-1} - \cdots - ||A_0||$ 
be a complex polynomial. Then $P(\lambda)$ is stable with respect to 
$\left(\mathbb{H} \setminus \overline{B(0;R)} \right) \subseteq \mathbb{H}$, where 
$R$ is the unique positive real zero of $u(z)$.
\end{theorem}

\begin{proof}
Consider the complex adjoint matrix polynomial 
$P_\chi(\lambda) = \bigchi_{A_m} \lambda^m + \cdots+ \bigchi_{A_1} \lambda + \bigchi_{A_0}$ of $P(\lambda)$. Invertibility of $A_m$ and $A_0$ implies 
$\bigchi_{A_m}$ and $\bigchi_{A_0}$ are invertible. Therefore, by Lemma $3.1$ of \cite{Higham-Tisseur}, $P_{\chi}(\lambda)$ is stable with respect to 
$D = \{z \in \mathbb{C}: |z| > R\}$, where $R$ is the unique positive real zero of the 
complex polynomial 
$||\bigchi_{A_m}^{-1}||^{-1} z^m - ||\bigchi_{A_{m-1}}|| z^{m-1} - \cdots - ||\bigchi_{A_0}|| = 
||A_m^{-1}||^{-1} z^m - ||A_{m-1}|| 
z^{m-1} - \cdots - ||A_0|| = u(z)$. That is, $P_\chi(\lambda)$ is stable with respect to 
$\left(\mathbb{H} \setminus \overline{B(0;R)} \right) \bigcap \mathbb{C}$. 
From Theorem \ref{Thm-relation bw matrix polynomial and its complex adjoint}, we conclude that $P(\lambda)$ is stable with respect to 
$\left(\mathbb{H} \setminus \overline{B(0;R)} \right) \bigcap \mathbb{C}$. Finally, 
the required conclusion follows from Lemma \ref{Lem-stability on complement of ball}.
\end{proof}

\medskip
We thus have the following theorem as an application of Theorems 
\ref{Thm-lower bound for eigenvalue} and \ref{Thm-upper bound for eigenvalue}. 
We state this below.
	
\begin{theorem}\label{H-T thm}
Right eigenvalues of a quaternion matrix polynomial $P(\lambda)$ with invertible leading 
and constant coefficients lie in the set $\{q \in \mathbb{H} : r \leq |q| \leq R \}$, where 
$r$ and $R$ are as given in Theorems \ref{Thm-lower bound for eigenvalue} and \ref{Thm-upper bound for eigenvalue} respectively.
\end{theorem}

\begin{remark}
In \cite{Basavaraju-Hadimani-Jayaraman-3}, 
the authors prove that if $P(\lambda)$ is a quaternion matrix polynomial with unitary 
matrix coefficients, then the eigenvalues of $P(\lambda)$ lie in the set 
$\{q \in \mathbb{H} : \frac{1}{2} \leq |q| \leq 2 \}$. This result was proved while 
studying the Hoffman-Wielandt inequality for quaternion matrix polynomials.
\end{remark}

\medskip

As the following example illustrates, both the lower and upper 
bounds $r$ and $R$ are attained.	
\begin{example}\label{ex-2.2}
In this example, both the lower and the upper bounds obtained from Theorems 
\ref{Thm-lower bound for eigenvalue} and \ref{Thm-upper bound for eigenvalue} 
coincide with the minimum and the maximum of the moduli of eigenvalues respectively. 
Let $P(\lambda)=A_2 \lambda^2 +A_1 \lambda + A_0$ where $A_2=A_0 = 
\begin{bmatrix}
		1 & 0 \\
		0 & 1 
\end{bmatrix}$ and $A_1=\begin{bmatrix}
		i & 0 \\
		0 & j
\end{bmatrix}$. It is easy to verify that the minimum and the maximum of the moduli of eigenvalues of $P(\lambda)$ are $\displaystyle \frac{-1+\sqrt{5}}{2}$ and 
$\displaystyle \frac{1+\sqrt{5}}{2}$ respectively. From Theorem 
\ref{Thm-lower bound for eigenvalue}, the lower bound on the set of moduli of eigenvalues 
is the unique positive zero of the complex polynomial 
$l(z) = ||A_2||z^2+ ||A_1|| z -||A_0^{-1}||^{-1}=z^2+z-1$ given by 
$r=\displaystyle \frac{-1+\sqrt{5}}{2}$.  Similarly, from 
Theorem \ref{Thm-upper bound for eigenvalue}, the upper bound is the unique 
positive zero of the complex polynomial 
$u(z) = ||A_2^{-1}||^{-1} z^2- ||A_1|| z -||A_0||=z^2-z-1$ given by 
$R=\displaystyle \frac{1+\sqrt{5}}{2}$.
\end{example}   

\medskip

We now bring out yet another application of the results obtained in this section. Recall that 
$A \in M_n(\mathbb{H})$ is said to be positive semidefinite (positive definite) if for every 
$x \in \mathbb{H}^n, \ x^{\ast}Ax \geq 0$ ($x^{\ast}Ax > 0$ for all $0 \neq x \in \mathbb{H}^n$). It is customary to denote positive semidefiniteness and positive 
definiteness by the notations $A \succeq 0$ and $A \succ 0$ respectively. It turns out that 
a quaternion matrix $A$ is positive (semi)definite if and only if its complex 
adjoint matrix is positive (semi)definite (see Proposition $3.4.1$, \cite{Rodman}). It 
follows from the definition that any right eigenvalue of a positive (semi)definite matrix is 
nonnegative(positive). One of the classical results on the location of the roots of complex polynomials inside the unit disc or more generally an annulus is the Enestr{\"o}m-Kakeya theorem. There is a vast literature on this result as one can infer from MathSciNet. In Theorem $2.3$ of \cite{Le-Du-Nguyen}, the authors prove a generalization of this for complex matrix polynomials. We state this below. 

\begin{theorem}\label{E-K them C}
Let $P(\lambda) = A_m \lambda^m + \dots + A_1 \lambda + A_0$ be a 
matrix polynomial whose coefficients $A_i \in M_n(\mathbb{C})$ satisfy $A_m \succeq 
A_{m-1} \succeq \dots \succeq A_0 \succeq 0; A_m \succ 0$. Then each eigenvalue 
$\lambda_0$ of $P(\lambda)$ satisfies $\frac{\lambda_{\text{min}}(A_0)}{2 \lambda_{\text{max}}(A_m)} 
\leq |\lambda_0| \leq 1$, where $\lambda_{\text{min}}(A_0)$ denotes the smallest 
eigenvalue of $A_0$ and $\lambda_{\text{max}}(A_m)$ denotes the largest eigenvalue 
of $A_m$.
\end{theorem}

We now derive an analogue of the above theorem for quaternion matrix polynomials.

\begin{theorem}\label{E-K thm Q}
Let $P(\lambda) = A_m \lambda^m + \dots + A_1 \lambda + A_0$ be an $n \times n$ 
quaternion matrix polynomial whose coefficients $A_i \in M_n(\mathbb{H})$ satisfy 
$A_m \succeq A_{m-1} \succeq \dots \succeq A_0 \succ 0$. Then each right eigenvalue 
$\lambda_0$ of $P(\lambda)$ satisfies 
$\frac{\lambda_{\text{min}}(A_0)}{2 \lambda_{\text{max}}(A_m)}
\leq |\lambda_0| \leq 1$, where $\lambda_{\text{min}}(A_0)$ denotes the smallest 
right eigenvalue of $A_0$ and $\lambda_{\text{max}}(A_m)$ denotes the largest right eigenvalue of $A_m$.	
\end{theorem}

\begin{proof}
Observe that $\frac{\lambda_{\text{min}}(A_0)}{2 \lambda_{\text{max}}(A_m)} =  \frac{\lambda_{\text{min}}(\bigchi_{A_0})}{2 \lambda_{\text{max}}(\bigchi_{A_m})}$. 
Let $\Omega_1$ and $\Omega_2$ denote respectively the open and closed balls in $\mathbb{H}$ of radii 
$\frac{\lambda_{\text{min}}(A_0)}{2 \lambda_{\text{max}}(A_m)}$ and $1$. 
Positive (semi)definiteness of the matrix coefficients (in $\mathbb{H}$) imply positive (semi)definiteness of the respective complex adjoint matrices; moreover the latter are 
ordered in the same way as the $A_i$'s. From Theorem \ref{E-K them C}, it follows that 
the complex adjoint matrix polynomial $P_{\chi}(\lambda)$ of $P(\lambda)$ is stable 
with respect to $\Omega_1 \cap \mathbb{C}$ and 
$\Big(\mathbb{H} \setminus \Omega_2\Big) \cap \mathbb{C}$. Resorting to Theorems
\ref{Thm-relation bw matrix polynomial and its complex adjoint} and 
\ref{Thm-stability on ball}, we deduce that the quaternion matrix polynomial 
$P(\lambda)$ is stable with respect to $\Omega_1$. From Proposition 
\ref{Prop-stability and eigenvalues}, 
we infer that any right eigenvalue $\lambda_0$ of $P(\lambda)$ satisfies 
$\frac{\lambda_{\text{min}}(A_0)}{2 \lambda_{\text{max}}(A_m)} \leq |\lambda_0|$. 
A similar argument yields that the quaternion matrix polynomial $P(\lambda)$ is stable 
with respect to $\mathbb{H} \setminus \Omega_2$ and therefore any right eigenvalue $\lambda_0$ of $P(\lambda)$ is also of absolute value at most $1$. This completes the 
proof of the theorem. 
\end{proof}

\subsection{Class of matrix polynomials for which stability implies hyperstability} 
\hspace*{\fill}\label{subsec-4.3}

We now proceed to identify classes of quaternion matrix polynomials for which stability 
implies hyperstability. Our first result in this direction concerns block upper triangular 
matrix polynomials, where each of the diagonal blocks are hyperstable with respect to 
a common set $\Omega$.

\begin{theorem}\label{Thm-upper triangular matrix polynomial}
Let $P(\lambda)$ be a block upper triangular quaternion matrix 
polynomial and $\Omega \subseteq \mathbb{H}$ be nonempty. If each of the diagonal 
blocks of $P(\lambda)$ are hyperstable with respect to $\Omega$, then $P(\lambda)$ is hyperstable with respect to $\Omega$.
\end{theorem}

\begin{proof}
Let $P(\lambda)$ be an $n \times n$ quaternion matrix polynomial of 
degree $m$ and of the form 
\begin{equation*}
P(\lambda) = \begin{bmatrix}
P_{11}(\lambda) & P_{12}(\lambda) & \cdots & P_{1t}(\lambda) \\
0 & P_{22}(\lambda) & \cdots & P_{2t}(\lambda) \\
\vdots & \vdots & \ddots & \vdots \\
0 & 0 & \cdots & P_{tt}(\lambda)
\end{bmatrix}, 
\end{equation*}
where $P_{ij}(\lambda) = \displaystyle \sum_{l=0}^{m}A_l^{(ij)}\lambda^l$ are matrix 
polynomials of size $k_i \times k_j$ for $1 \leq i,j \leq t$ and 
$\displaystyle \sum_{j=1}^{t} k_j =n$. Assume that $P_{ii}(\lambda)$'s are hyperstable 
with respect to $\Omega$ for $1 \leq i \leq t$. Consider a nonzero vector 
$y =\begin{bmatrix}
y_1^T & y_2^T & \cdots & y_t^T
\end{bmatrix}^T \in \mathbb{H}^n$, where $y_i \in \mathbb{H}^{k_i}$, for 
$1 \leq i \leq t$. Let $r$ denote the index of the last nonzero $y_i$. Since 
$P_{rr}(\lambda)$ is hyperstable with respect to $\Omega$, there exists 
$z_r \in \mathbb{H}^{k_r} \setminus \{0\}$ such that for all $\lambda \in \Omega$ 
\begin{equation*}
z_r^*A_m^{(rr)}y_r \lambda^m + \cdots + z_r^*A_1^{(rr)}y_r \lambda + z_r^*A_0^{(rr)}y_r \neq 0. 
\end{equation*} 
Set $z:= \begin{bmatrix}
0 & \cdots  & z_r^T  & \cdots & 0 
\end{bmatrix}^T \in \mathbb{H}^n \setminus \{0\}$. On expanding $P(\lambda)$, 
we have
\begin{equation}
P(\lambda) = \begin{bmatrix}
\displaystyle \sum_{l=0}^{m}A_l^{(11)}\lambda^l & \displaystyle \sum_{l=0}^{m}A_l^{(12)}\lambda^l 
& \cdots & \displaystyle \sum_{l=0}^{m}A_l^{(1t)}\lambda^l\\
0 & \displaystyle \sum_{l=0}^{m}A_l^{(22)}\lambda^l & \cdots & \displaystyle 
\sum_{l=0}^{m}A_l^{(2t)}\lambda^l\\
\vdots & \vdots & \ddots & \vdots \\
0 & 0 & \cdots & \displaystyle \sum_{l=0}^{m}A_l^{(tt)}\lambda^l
\end{bmatrix} = \displaystyle \sum_{l=0}^{m}A_l\lambda^l,  
\end{equation}
where $A_l = \begin{bmatrix}
A_l^{(11)} & A_l^{(12)} & \cdots & A_l^{(1t)} \\
0 & A_l^{(22)} & \cdots & A_l^{(2t)} \\
\vdots & \vdots & \ddots & \vdots \\
0 & 0 & 0 & A_l^{(tt)} 
\end{bmatrix}$ for $1 \leq l \leq m$. For $1 \leq l \leq m$, consider 
\begin{equation}
z^*A_ly = \begin{bmatrix}
0 & \cdots & z_r^* & \cdots & 0
\end{bmatrix} \begin{bmatrix}
A_l^{(11)} & A_l^{(12)} & \cdots & A_l^{(1t)} \\
0 & A_l^{(22)} & \cdots & A_l^{(2t)} \\
\vdots & \vdots & \ddots & \vdots \\
0 & 0 & 0 & A_l^{(tt)} 
\end{bmatrix} \begin{bmatrix}
y_1 \\
\vdots \\
y_r \\
\vdots \\
0
\end{bmatrix} = z_r^*A_l^{(rr)}y_r. 
\end{equation}
We thus have $\displaystyle \sum_{l=0}^{m}z^*A_ly \lambda^i = \sum_{l=0}^{m} 
z_r^*A_l^{(rr)}y_r \lambda^i$. Since $P_{rr}(\lambda)$ is hyperstable with respect 
to $\Omega$, for any $\lambda \in \Omega$, 
$\displaystyle \sum_{l=0}^{m}z^*A_ly \lambda^i = 
\displaystyle \sum_{l=0}^{m} z_r^*A_l^{(rr)}y_r \lambda^i \neq 0$. 
This proves that $P(\lambda)$ is hyperstable with respect to $\Omega$. 
\end{proof} 

\medskip
\noindent
We now discuss equivalence of stability and hyperstability for quaternion polynomials. 
This will be used in the theorem that follows.

\medskip
\begin{theorem}\label{Thm-quaternion polynomial}
Let $p\colon \mathbb{H} \rightarrow \mathbb{H}$ be a quaternion scalar polynomial 
and $\Omega$ be any subset of $\mathbb{H}$. Then $p(\lambda)$ is hyperstable with 
respect to $\Omega$ if and only if $p(\lambda)$ is stable with respect to $\Omega$.  
\end{theorem}

\begin{proof}
Let $p(\lambda) = a_m \lambda^m + \cdots +a_1\lambda +a_0$, where 
$a_i \in \mathbb{H}$, for $i = 0,1,\ldots, m$. If $p(\lambda)$ is hyperstable with 
respect to $\Omega$, then certainly, it is stable with respect to $\Omega$ by definition. 
Suppose $p(\lambda)$ is stable with respect to $\Omega$. Let 
$y \in \mathbb{H} \setminus \{0\}$ be arbitrary. Taking $z =1$, we have 
$z^* a_m y \mu^m + \cdots + z^* a_1 y \mu +z^* a_0 y \neq 0$ 
for all $\mu \in \Omega$. For otherwise, 
$a_m y \mu^m_0 + \cdots + a_1 y \mu_0 + a_0 y =0$ for some $\mu_0 \in \Omega$. 
Since $y \neq 0, \ \mu_0$ is an eigenvalue of $p(\lambda)$, a contradiction to 
Proposition \ref{Prop-stability and eigenvalues}. Therefore, 
$z^* a_m y \mu^m + \cdots + z^* a_1 y \mu +z^* a_0 y \neq 0$ for all $\mu \in \Omega$. 
This proves hyperstability of $p(\lambda)$ with respect to $\Omega$. 
\end{proof}

\medskip
\begin{remark}
Note that, in the complex case, a scalar is an eigenvalue of a scalar polynomial if and only 
if it is a zero of that polynomial. In the quaternion case, this condition 
does not generally hold. For example, consider $p(\lambda) = \lambda - k$. It is easy 
to verify that $-k$ is an eigenvalue of $p(\lambda)$ with eigenvector $y = i+j$, but not a 
zero of $p(\lambda)$. However, we have the following result for quaternion scalar 
polynomials.
\end{remark}

\medskip
\begin{lemma}\label{Lem-eigenvalue of scalar polynomial}
Let $p(\lambda) = \lambda^m + a_{m-1} \lambda^{m-1} + \cdots +a_1 \lambda +a_0$ 
be a quaternion monic scalar polynomial. Then $\lambda_0$ is an eigenvalue of 
$p(\lambda)$ if and only if $\lambda_0$ is a zero of the scalar polynomial 
$\displaystyle \sum_{k=1}^{2m} \sum_{i+j=k} a_i \overline{a_j} \lambda^k$.
\end{lemma}

\begin{proof}
Consider the companion matrix $A = \begin{bmatrix}
0 & 1 & \cdots & 0 \\
\vdots & \vdots & \ddots & \vdots \\
0 & 0 & \cdots & 1 \\
-a_0 & -a_1 & \cdots & -a_{m-1} 
\end{bmatrix}$ of $p(\lambda)$. Theorem $5.1$ of \cite{Ahmad-Ali-1} tells us 
that $\lambda_0$ is an eigenvalue of $p(\lambda)$ if and only if $\lambda_0$ is an 
eigenvalue of $A$. The desired conclusion follows from Theorem $4.2$ of 
\cite{Pereira-Rocha}.
\end{proof}

\medskip
\noindent
We are now in a position to prove the main result of this section.

\medskip
\begin{theorem}\label{Thm-equivalence of stability and hyperstability}
Let $P(\lambda)$ be as in \eqref{eq-matrix polynomial} with leading coefficient to 
be the identity matrix, and let $\Omega \subseteq \mathbb{H}$ be nonempty. 
If $P(\lambda)$ is an upper-triangular matrix polynomial, then $P(\lambda)$ is 
hyperstable with respect $\Omega$ if and only if it is stable with respect to $\Omega$.
\end{theorem}

\begin{proof}
Let $P(\lambda) = \begin{bmatrix}
p_{11}(\lambda) & p_{12}(\lambda)& \cdots & p_{1n}(\lambda) \\
0 & p_{22}(\lambda) & \cdots & p_{2n}(\lambda) \\
\vdots & \vdots & \ddots & \vdots \\
0 & 0 & \cdots & p_{nn}(\lambda)
\end{bmatrix} = I \lambda^m + A_{m-1} \lambda^{m-1} + \cdots + A_1 \lambda +A_0$, 
where $p_{ij}(\lambda) = \lambda^m + a_{m-1}^{(ij)} \lambda^{m-1} + 
\cdots+ a^{(ij)}_1 \lambda + a^{(ij)}_0$ 
are quaternion monic scalar polynomials and $A_i \in M_n(\mathbb{H})$ are 
upper-triangular matrices. Suppose $P(\lambda)$ is stable with respect to $\Omega$. 
We then claim that $p_{ll}(\lambda)$ are stable with respect to $\Omega$ for all $l=1,2,\ldots,n$. 
If $p_{ll}(\lambda)$ is not stable with respect $\Omega$ for some $l=1,2,\ldots,n$, 
then by Proposition \ref{Prop-stability and eigenvalues}, there exists an eigenvalue 
$\lambda_0 \in \Omega$ of $p_{ll}(\lambda)$. From Lemma 
\ref{Lem-eigenvalue of scalar polynomial}, it follows that $\lambda_0$ is a 
zero of the scalar polynomial 
$ \displaystyle \sum_{k=1}^{2m} \sum_{i+j=k} a^{(ll)}_i \overline{a_j^{(ll)}} \lambda^k$. 
This implies that $\lambda_0$ is a zero of the scalar polynomial 
$\displaystyle \prod_{l=1}^{n} \left(\sum_{k=1}^{2m} \sum_{i+j=k} a^{(ll)}_i 
\overline{a_j^{(ll)}} \lambda^k \right)$. From Corollary $4.3$ of \cite{Pereira-Rocha},  $\lambda_0$ is an eigenvalue of the block matrix 
$C = \begin{bmatrix}
		0 & I & 0 & \cdots & 0 \\
		0 & 0 & I & \cdots & 0 \\
		\vdots & \vdots & \vdots & \ddots & \vdots \\
		0 & 0 & 0 & \cdots & I \\
		-A_0 & -A_1 & -A_2 & \cdots & -A_{m-1} 
\end{bmatrix}$. Thus, by Theorem $5.1$ of \cite{Ahmad-Ali-1}, $\lambda_0$ is 
an eigenvalue of $P(\lambda)$, thereby proving that $P(\lambda)$ is not stable with 
respect to $\Omega$. This contradiction proves that $p_{ll}(\lambda)$ are stable with 
respect to $\Omega$ for all $l=1,2,\ldots,n$. From Theorem 
\ref{Thm-quaternion polynomial}, we then infer that $p_{ll}(\lambda)$ are hyperstable 
with respect to $\Omega$ for all $l =1,2,\ldots,n$. The desired conclusion then 
follows from Theorem \ref{Thm-upper triangular matrix polynomial}. The converse 
follows from the definition.
\end{proof}

\begin{remark}
In Theorem $3.7$ of \cite{Oskar-Wojtylak}, the authors prove that stability and 
hyperstability are equivalent for complex matrix polynomials of the form 
$P(\lambda) = p(\lambda)A+q(\lambda)B$ with some scalar complex polynomials $p(\lambda)$, $q(\lambda)$ and $A, B \in M_n(\mathbb{C})$. The proof uses 
the well known generalized Schur form. However, due to noncommutativity of quaternions 
this is not amenable in the quaternion setting unless the variable itself is real and $A$ and $B$ 
are complex matrices. 
We wish to point out that Theorem \ref{Thm-equivalence of stability and hyperstability} 
is one interesting result which avoids this classical technique and noncommutativity of quaternions comes into picture.  
\end{remark}

\medskip

\subsection{Stability and hyperstability of multivariate quaternion matrix polynomials} 
\hspace*{\fill}\label{subsec-4.4}

We now define multivariate quaternion matrix polynomials and extend the notions of 
stability and hyperstability for the same. These are analogous to the 
definitions given in \cite{Oskar-Wojtylak} for multivariate complex matrix 
polynomials. In particular, we establish hyperstability of 
univariate quaternion matrix polynomials from stability of multivariate quaternion matrix 
polynomials. Although some of the results are very similar to those from 
\cite{Oskar-Wojtylak} in the complex case and are amenable for generalization to the 
quaternion setting, not everything goes through due to noncommutativity of the 
variables. We begin with the definitions and present some examples to illustrate the same.   

\medskip
\begin{definition}\label{Def-multivariate matrix polynomial}
A multivariate right quaternion matrix polynomial in $k$ variables of degree 
$m$ and size $n$ is a function 
$P(\lambda_1, \lambda_2, \ldots, \lambda_k) \colon \mathbb{H}^k 
\rightarrow M_n(\mathbb{H})$ defined by 
\begin{equation}\label{Eqn-multivariable polynomial}
P(\lambda_1, \lambda_2, \ldots, \lambda_k) = \displaystyle \sum_{w \in W} A_{w} 
w(\lambda_1, \lambda_2, \ldots, \lambda_k), 
\end{equation}
where $W$ is the set of all possible words of finite length formed from the 
noncommuting letters $\{ \lambda_1, \lambda_2, \ldots, \lambda_k \}$ such that 
the sum of the powers does not exceed $m$.
\end{definition}

\medskip
\begin{definition}\label{Def-multivariate-stability}
Let $\Omega$ be a subset of $\mathbb{H}$. We say that the multivariate right 
quaternion matrix polynomial $P(\lambda_1, \lambda_2, \ldots, \lambda_k) = 
\displaystyle \sum_{w \in W} A_{w} w(\lambda_1, \lambda_2, \ldots, \lambda_k)$ is 
stable with respect to $\Omega^k$, if for any $y \in \mathbb{H}^n \setminus \{0\}$ 
and $ (\mu_1, \mu_2, \ldots, \mu_k) \in \Omega^k$, there exists $z \in \mathbb{H}^n 
\setminus \{0\}$ such that 
\begin{equation}\label{Eqn-multivariable stability}
\displaystyle \sum_{w \in W} z^\ast A_{w}y w(\mu_1, \mu_2, \ldots, \mu_k) \neq 0.
\end{equation}
\end{definition}

\medskip

\begin{definition}\label{Def-multivariate-hyperstability}
Let $\Omega$ be a subset of $\mathbb{H}$. We say that the multivariate right 
quaternion matrix polynomial $P(\lambda_1, \lambda_2, \ldots, \lambda_k) = 
\displaystyle \sum_{w \in W} A_{w} w(\lambda_1, \lambda_2, \ldots, \lambda_k)$ 
is hyperstable with respect to $\Omega^k$, if for any 
$y \in \mathbb{H}^n \setminus \{0\}$, there exists 
$z \in \mathbb{H}^n \setminus \{0\}$ such that 
\begin{equation}\label{Eqn-multivariable hyperstability}
\displaystyle \sum_{w \in W} z^\ast A_{w}y w(\mu_1, \mu_2, \ldots, \mu_k) \neq 0 
\hspace{0.3cm} 
\text{for all} \hspace{0.2cm} (\mu_1, \mu_2, \ldots, \mu_k) \in \Omega^k.
\end{equation}
\end{definition}

\medskip
\noindent
In the following two examples, we illustrate these definitions for multivariate right 
quaternion matrix polynomials.

\begin{example}\label{ex-3}
Let $P(\lambda_1, \lambda_2) = \begin{bmatrix}
1 & 0 \\
0 & 1
\end{bmatrix} \lambda_1 \lambda_2 + \begin{bmatrix}
1 & 0 \\
0 & 1
\end{bmatrix} \lambda_2$ and $\Omega \subseteq \mathbb{H}$ such that 
$0, -1 \notin \Omega$. Let us take an element 
$y = \begin{bmatrix}
y_1 \\
y_2
\end{bmatrix} \in \mathbb{H}^2 \setminus \{0\}$ and let 
$(\mu_1,\mu_2) \in \Omega^2$ be arbitrary. Consider 
\begin{align*}
\begin{bmatrix}
1 & 0 \\
0 & 1
\end{bmatrix} \begin{bmatrix}
y_1 \\
y_2
\end{bmatrix} \mu_1 \mu_2 + \begin{bmatrix}
1 & 0 \\
0 & 1
\end{bmatrix} \begin{bmatrix}
y_1 \\
y_2
\end{bmatrix} \mu_2 & = \begin{bmatrix}
y_1 (\mu_1 +1)\mu_2 \\
y_2(\mu_1 + 1)\mu_2
\end{bmatrix}.
\end{align*} 
If $y_1 \neq 0$, choose $z= \begin{bmatrix}
1 \\
0
\end{bmatrix}$. Then, we have 
\begin{align*}
& \begin{bmatrix}
1 & 0
\end{bmatrix}
\begin{bmatrix}
1 & 0 \\
0 & 1
\end{bmatrix} \begin{bmatrix}
y_1 \\
y_2
\end{bmatrix} \mu_1 \mu_2 + \begin{bmatrix}
1 & 0
\end{bmatrix} \begin{bmatrix}
1 & 0 \\
0 & 1
\end{bmatrix} \begin{bmatrix}
y_1 \\
y_2
\end{bmatrix} \mu_2 \\
& = y_1 (\mu_1 +1)\mu_2 \neq 0 \, \text{for any} \,(\mu_1, \mu_2) \in \Omega^2. 
\end{align*} 
 If $y_2 \neq 0$, choose $z = \begin{bmatrix}
0 \\
1
\end{bmatrix}$. Then, we have 
\begin{align*}
& \begin{bmatrix}
0 & 1
\end{bmatrix}
\begin{bmatrix}
1 & 0 \\
0 & 1
\end{bmatrix} \begin{bmatrix}
y_1 \\
y_2
\end{bmatrix} \mu_1 \mu_2 + \begin{bmatrix}
0 & 1
\end{bmatrix} \begin{bmatrix}
1 & 0 \\
0 & 1
\end{bmatrix} \begin{bmatrix}
y_1 \\
y_2
\end{bmatrix} \mu_2 \\
& = y_2 (\mu_1 +1)\mu_2 \neq 0 \, \text{for any} \, (\mu_1, \mu_2) \in \Omega^2. 
\end{align*} 
Thus, in either of the cases, $P(\lambda_1, \lambda_2)$ is hyperstable with 
respect to $\Omega^2$.
\end{example}

\medskip
\noindent
Here is an example of a matrix polynomial which is not stable with respect to a given 
set.

\medskip
\begin{example}\label{ex-4}
Let $P(\lambda_1, \lambda_2) = \begin{bmatrix}
1 & 0 \\
0 & 1
\end{bmatrix} \lambda_1 \lambda_2 + \begin{bmatrix}
1 & 0 \\
0 & 0
\end{bmatrix} \lambda_2 \lambda_1 + \begin{bmatrix}
1 & 0 \\
0 & 1
\end{bmatrix}\lambda_1 + \begin{bmatrix}
1 & 0 \\
0 & 1
\end{bmatrix}$ and $\Omega = B(0;1) \subseteq \mathbb{H}$. Consider $y = 
\begin{bmatrix}
1 \\
0
\end{bmatrix} \in \mathbb{H}^2$ and 
$(\mu_1, \mu_2) = (-\frac{1}{2}, \frac{1}{2}) \in \Omega^2$. Then 
\begin{align*}
& \begin{bmatrix}
1 & 0 \\
0 & 1
\end{bmatrix} 
\begin{bmatrix}
1 \\
0
\end{bmatrix} \mu_1 \mu_2 + \begin{bmatrix}
1 & 0 \\
0 & 0
\end{bmatrix} \begin{bmatrix}
1 \\
0
\end{bmatrix} \mu_2 \mu_1 + \begin{bmatrix}
1 & 0 \\
0 & 1
\end{bmatrix} \begin{bmatrix}
1 \\
0
\end{bmatrix}\mu_1 + \begin{bmatrix}
1 & 0 \\
0 & 1
\end{bmatrix} \begin{bmatrix}
1 \\
0
\end{bmatrix} \\
& = \begin{bmatrix}
1\\
0
\end{bmatrix} \bigg( -\frac{1}{4} \bigg) + \begin{bmatrix}
1\\
0
\end{bmatrix} \bigg( -\frac{1}{4} \bigg) + \begin{bmatrix}
1\\
0
\end{bmatrix} \bigg( -\frac{1}{2} \bigg) +\begin{bmatrix}
1\\
0
\end{bmatrix} \\
& = \begin{bmatrix}
-\frac{1}{4}-\frac{1}{4}-\frac{1}{2} + 1\\
0
\end{bmatrix} = \begin{bmatrix}
0 \\
0
\end{bmatrix}.
\end{align*}
Therefore, for the above chosen $y$ and $(\mu_1, \mu_2) \in \Omega^2$, there 
exists no $z \in \mathbb{H}^2 \setminus \{0\}$ such that Equation 
\eqref{Eqn-multivariable stability} holds. Therefore, $P(\lambda_1, \lambda_2)$ 
is not stable with respect to $\Omega^2$.
\end{example}

\medskip
\noindent
We present some sufficient conditions for hyperstability 
of quaternion matrix polynomials in one variable via stability of multivariate 
quaternion matrix polynomials. 

\medskip
\begin{theorem}\label{multivariate-univariate-1}
Let $P(\lambda)= A_2 \lambda^2 + A_1 \lambda + A_0$ be a quaternion matrix 
polynomial and $\Omega \subseteq \mathbb{H}$ be nonempty. Then $P(\lambda)$ is 
hyperstable with respect to $\Omega$ if any one of the following holds:
\begin{itemize}
\item[(i)] $P(\lambda_1, \lambda_2) = A_2 \lambda_1^2 + A_1 \lambda_2 +A_0$ is 
stable with respect to $\Omega^2$.
\item[(ii)] $P(\lambda_1, \lambda_2) = A_2 \lambda_1 \lambda_2 + A_1 \lambda_2 +A_0$ 
is stable with respect to $\Omega^2$ and $0 \notin \Omega$.
\end{itemize}
\end{theorem}

\begin{proof}
We present here only the proof of $(i)$, as the proof of $(ii)$ is similar. Suppose 
$P(\lambda_1, \lambda_2) = A_2 \lambda_1^2 + A_1 \lambda_2 +A_0$ is 
stable with respect to $\Omega$. Note that, for $\lambda_1 = \lambda_2 = \lambda$, we 
have $P(\lambda_1, \lambda_2) = P(\lambda)$. Therefore, $P(\lambda)$ is stable with 
respect to $\Omega$. Let $y \in \mathbb{H}^n \setminus \{0\}$ be arbitrary but fixed. 
If there is a $z \in \mathbb{H}^n \setminus \{0\}$ such that $\langle A_2y, z \rangle =0, 
\langle A_1y, z \rangle = 0$ and $\langle A_0y, z\rangle \neq 0$, then 
\begin{center}
$z^* A_2y \lambda^2 + z^* A_1 y \lambda + z^* A_0y = z^*A_0 y \neq 0$ 
for any $\lambda \in \Omega$. 
\end{center}
Therefore $P(\lambda)$ is hyperstable with respect to $\Omega$.\\
If there is no $z \in \mathbb{H}^n \setminus \{0\}$ such that 
$\langle A_2y,z \rangle =0, \langle A_1y, z \rangle = 0$ and 
$\langle A_0y, z \rangle \neq 0$, then 
$A_0y \in \textit{span}_{\mathbb{H}}\{A_2y, A_1y\}$, so that 
$A_0y = A_2y q_2 +A_1yq_1$ for some $q_1, q_2 \in \mathbb{H}$. Then, we have 
\begin{equation}\label{Eqn-*}
A_2 y \lambda^2 + A_1 y \lambda + A_0 y = A_2 y (\lambda^2 +q_2) +A_1 y (\lambda + q_1). 
\end{equation}
Suppose $\mu_1 , \mu_2 \in \Omega$ are zeros of $\lambda^2 + q_2$ and $\lambda + q_1$ 
respectively. Then, $ A_2 y \mu_1^2 +A_1 y \mu_2 + A_0 y = A_2 y \mu_1^2 +A_1 y \mu_2 + 
A_2y q_2 + A_1y q_1 = A_2 y (\mu_1^2 +q_2) +A_1 y (\mu_2 + q_1) = 0$. This implies 
$P(\lambda_1, \lambda_2)$ is not stable with respect to $\Omega^2$, which is a 
contradiction to our assumption. Therefore, atleast one of $\lambda^2 +q_2$ and $\lambda+q_1$ 
do not have zeros in $\Omega$. \\
Case $(1)$: Assume that $\lambda^2 +q_2$ does not have zeros in $\Omega$. \\
If there is a vector $z \in \mathbb{H}^n \setminus \{0\}$ such that 
$\langle A_1y, z \rangle = 0$ and $\langle A_2 y, z \rangle \neq 0$, then 
\begin{align*}
z^* A_2 y \lambda^2 + z^* A_1 y \lambda + z^* A_0 y & = z^* A_2y(\lambda^2 + q_1) + 
z^* A_1 y (\lambda + q_1) \hspace{0.5cm} \text{by} \hspace{0.2cm} \eqref{Eqn-*} \\
& = z^*A_2y (\lambda^2 +q_2) \\
& \neq 0 \hspace{0.3cm} \text{for all} \hspace{0.1cm} \lambda \in \Omega. 
\end{align*}
This implies $P(\lambda)$ is hyperstable with respect to $\Omega$. \\
Suppose there is no vector $z \in \mathbb{H}^n \setminus \{0\}$ such that 
$\langle A_1y, z \rangle = 0$ and $\langle A_2 y, z \rangle \neq 0$. Then 
$A_2y \in \textit{span}_{\mathbb{H}}\{A_1y\}$ and $A_2y = A_1y q_3$ for some 
$q_3 \in \mathbb{H}$. Substituting in \eqref{Eqn-*}, 
we get 
\begin{equation}\label{Eqn-**}
A_2 y \lambda^2 + A_1 y \lambda + A_0 y = A_1y (q_3 \lambda^2 + \lambda + q_3 q_2 +q_1). 
\end{equation}
Since $P(\lambda)$ is stable with respect to $\Omega$, we have $A_1y \neq 0$ and 
$q_3 \lambda^2 + \lambda + q_3 q_2 +q_1 \neq 0$ for any $\lambda \in \Omega$. 
Taking $z = A_1y$, we have
\begin{align*}
z^* A_2 y \lambda^2 + z^* A_1 y \lambda + z^* A_0 y & = z^* A_1y (q_3 \lambda^2 + 
\lambda + q_3 q_2  +q_1) \hspace{0.5cm} \text{by} \hspace{0.2cm} \eqref{Eqn-**} \\
& = (A_1y)^* (A_1y) (q_3 \lambda^2 + \lambda + q_3 q_2 +q_1) \\
& \neq 0 \hspace{0.3cm} \text{for all } \lambda \in \Omega. 
\end{align*}
Therefore, $P(\lambda)$ is hyperstable with respect to $\Omega$.\\
Case(2): Assume that $\lambda + q_1$ does not have zeros in $\Omega$. \\
If there exists $z \in \mathbb{H}^n \setminus \{0\}$ such that 
$\langle A_2y, z \rangle = 0$ and $\langle A_1 y, z \rangle \neq 0$, then 
\begin{align*}
z^* A_2 y \lambda^2 + z^* A_1 y \lambda + z^* A_0 y & = z^* A_2y(\lambda^2 + q_1) 
+ z^* A_1 y (\lambda + q_1) \hspace{0.5cm} \text{by} \hspace{0.2cm} \eqref{Eqn-*} \\
& = z^*A_1y (\lambda + q_1) \\
& \neq 0 \hspace{0.3cm} \text{for all } \lambda \in \Omega. 
\end{align*}
This implies $P(\lambda)$ is hyperstable with respect to $\Omega$. Therefore, 
we may assume that there is no $z \in \mathbb{H}^n \setminus \{0\}$ such that 
$\langle A_2y, z \rangle = 0$ and $\langle A_1 y, z \rangle \neq 0$. This implies 
$A_1y \in \textit{span}_{\mathbb{H}}\{A_2y \}$, so that 
$A_1y = A_2y q_4$ for some $q_4 \in \mathbb{H}$. Substituting in \eqref{Eqn-*}, 
we get 
\begin{equation}\label{Eqn-***}
A_2 y \lambda^2 + A_1 y \lambda + A_0 y = A_2 y ( \lambda^2 + q_4 \lambda +q_2 + q_4 q_1). 
\end{equation}
Once again, since $P(\lambda)$ is stable with respect to $\Omega$, we see that 
$A_2y \neq 0$ and the scalar polynomial $\lambda^2 + q_4 \lambda +q_2 + q_4 q_1$ 
does not have any zeros in $\Omega$. Now by taking $z = A_2y$ we have, 
\begin{align*}
z^* A_2 y \lambda^2 + z^* A_1 y \lambda + z^* A_0 y & = z^* A_2 y ( \lambda^2 + q_4 
\lambda +q_2 + q_4 q_1) \hspace{0.5cm} \text{by} \hspace{0.2cm} \eqref{Eqn-***} \\
& = (A_2 y)^* (A_2 y) ( \lambda^2 + q_4 \lambda +q_2 +q_4 q_1) \\
& \neq 0 \hspace{0.3cm} \text{for all } \lambda \in \Omega. 
\end{align*}
Thus, $P(\lambda)$ is hyperstable with respect to $\Omega$.
\end{proof}

\medskip
\noindent
We end with the following theorem where a sufficient condition for hyperstability of 
a cubic matrix polynomial in one variable via stability of a multivariate quaternion 
matrix polynomial can be derived. We skip the proof as the proof technique is similar 
to that of the quadratic case presented in the previous theorem.

\medskip
\begin{theorem}
Let $P(\lambda) = A_0 \lambda^3 + A_2 \lambda^2 + A_1 \lambda +A_0$ be a cubic 
quaternion matrix polynomial. Let $\Omega \subseteq \mathbb{H}$ be nonempty and 
$0 \notin \Omega$. Then $P(\lambda)$ is hyperstable with respect to $\Omega$, 
if the multivariate right matrix polynomial 
$P(\lambda_1 , \lambda_2) = A_0 \lambda_2^3 + A_2 \lambda_1 \lambda_2 + 
A_1 \lambda_1 +A_0$ is stable with respect to $\Omega^2$.
\end{theorem}

\medskip
\noindent
We end the manuscript with the following remark.

\medskip
\begin{remark}
Note that in \cite{Oskar-Wojtylak}, the authors have considered other multivariate 
matrix polynomials as well to obtain hyperstability of a given univariate matrix 
polynomial. However, in the quaternion case, due to noncommutativity of variables 
one cannot conclude hyperstability of a given univariate matrix polynomial. 
\end{remark} 

\bibliographystyle{amsalpha}

\end{document}